\documentclass[11pt]{amsart}

\usepackage{kantlipsum} 
\calclayout

\usepackage{amsmath,amssymb,amscd}
\usepackage{upgreek}
\usepackage{graphicx}
\usepackage{wasysym,footnote}
\usepackage{amsthm}
\usepackage[all]{xy}
\usepackage{tikz}
\usetikzlibrary{arrows}
\usetikzlibrary{decorations.markings}
\input{quivers_HomE.sty}
\input{polytope.sty}
\usepackage{hyperref}
\DeclareMathOperator{\Aut}{Aut}
\DeclareMathOperator{\Coker}{Coker}

\DeclareMathOperator{\E}{E}

\DeclareMathOperator{\Ext}{Ext}
\DeclareMathOperator{\ext}{ext}
\DeclareMathOperator{\GL}{GL}
\DeclareMathOperator{\Gr}{Gr}
\DeclareMathOperator{\Hom}{Hom}
\DeclareMathOperator{\Id}{Id}

\DeclareMathOperator{\ind}{ind}

\DeclareMathOperator{\Ker}{Ker}

\DeclareMathOperator{\PHom}{PHom}
\DeclareMathOperator{\IHom}{IHom}

\DeclareMathOperator{\rep}{rep}
\DeclareMathOperator{\res}{res}
\DeclareMathOperator{\SI}{SI}
\DeclareMathOperator{\SL}{SL}

\DeclareMathOperator{\T}{T}

\newtheorem{theorem}{Theorem}[section]
\newtheorem{lemma}[theorem]{Lemma}
\newtheorem{corollary}[theorem]{Corollary}
\newtheorem{proposition}[theorem]{Proposition}

\theoremstyle{definition}
\newtheorem{definition}[theorem]{Definition}
\newtheorem*{definition*}{Definition}
\newtheorem{example}[theorem]{Example}
\newtheorem{conjecture}[theorem]{Conjecture}

\theoremstyle{remark}
\newtheorem{remark}[theorem]{Remark}
\newtheorem{observation}[theorem]{Observation}
\newtheorem{algorithm}[theorem]{Algorithm}

\newtheorem{question}[theorem]{Question}
\newtheorem*{question*}{Question}
\newtheorem{problem}[theorem]{Problem}

\numberwithin{equation}{section}

\newcommand{\op}[1]{\operatorname{#1}}
\newcommand{\mb}[1]{\mathbb{#1}}
\newcommand{\mc}[1]{\mathcal{#1}}
\newcommand{\mf}[1]{\mathfrak{#1}}

\newcommand{\bs}[1]{\boldsymbol{#1}}
\renewcommand{\b}[1]{\bold{#1}}

\newcommand{\e}{{\op{e}}}
\newcommand{\g}{{\sf g}}
\newcommand{\F}{{\sf F}}
\newcommand{\N}{{\sf N}}
\newcommand{\V}{{\sf V}}

\newcommand{\Ec}{{\check{\E}}}
\newcommand{\ec}{{\check{\e}}}
\newcommand{\Fc}{{\check{\sf \F}}}
\newcommand{\tc}{{\check{t}}}
\newcommand{\fc}{{\check{f}}}
\newcommand{\dc}{{\check{d}}}

\newcommand{\Nc}{{\check{\N}}}
\newcommand{\dtc}{{\check{\delta}}}
\newcommand{\dtb}{{\br{\delta}}}

\newcommand{\ep}{{\epsilon}}

\newcommand{\proj}{\operatorname{proj}\text{-}}

\newcommand{\br}[1]{\overline{#1}}
\newcommand{\innerprod}[1]{\langle#1\rangle}
\renewcommand{\ss}[2]{{\rep^{#1}({#2})}}
\newcommand{\sm}[1]{\left(\begin{smallmatrix}#1\end{smallmatrix}\right)}

\newcommand{\dv}{\underline{\dim}}

\newcommand{\wtd}[1]{\widetilde{#1}}

\newcommand{\bminus}{{\scriptscriptstyle \boxminus}}
\newcommand{\bplus}{{\scriptscriptstyle \boxplus}}
\newcommand{\pperp}{\perp\!\!\!\perp}

\begin{document}
	
\title{Tropical $F$-polynomials and General Presentations}
\author{Jiarui Fei}
\address{School of Mathematical Sciences, Shanghai Jiao Tong University}
\email{jiarui@sjtu.edu.cn}
\thanks{The author was supported in part by National Science Foundation of China (No. 11971305 and No. 12131015)}

% General info

\date{}
\dedicatory{}

\begin{abstract} We introduce the tropical $F$-polynomial $f_M$ of a quiver representation $M$. We study its interplay with the general presentation for any finite-dimensional basic algebra.
We give an interpretation of evaluating $f_M$ at a weight vector.
As a consequence, we give a presentation of the Newton polytope $\N(M)$ of $M$.
We study the dual fan and 1-skeleton of $\N(M)$.
We propose an algorithm to determine the generic Newton polytopes, and show it works for path algebras.
As an application, we give a representation-theoretic interpretation of Fock-Goncharov's duality pairing.
We give an explicit construction of dual clusters, which consists of real Schur representations.
We specialize the above general results to the cluster-finite algebras and the preprojective algebras of Dynkin type.
\end{abstract}

\maketitle
% !TeX spellcheck = en_GB 
\setcounter{tocdepth}{2}
\tableofcontents

\section{Introduction}

For fixed projective representations $P_-,P_+$ of a given path algebra with relations $A=kQ/I$, a {\em general presentation} $d:P_-\to P_+$ is a general element in the vector space $\Hom(P_-,P_+)$. 
The study of general presentations of algebras was initiated in \cite{DF}. 
The theory was developed in parallel with that of general representations of quivers (without relations) (e.g., \cite{Ka1,S2}).
The dimension vectors in our setting is replaced by the weight vectors $\delta\in \mb{Z}^{Q_0}$.
The presentation space of weight $\delta$ is the space 
$$\PHom(\delta):=\Hom \left(P([-\delta]_+),P([\delta]_+) \right),$$ 
where we denote $[\delta]_+ := \max(\delta,0)$, $P(\beta):=\bigoplus_{i\in Q_0} \beta(i)P_i$, and $P_i$ is the indecomposable projective representation corresponding to the vertex $i$.
For two presentations $d_1,d_2$, we defined a finite-dimensional space $\E(d_1,d_2)$ which plays the role of $\Ext^1$ for path algebras (without relations). 
We denote by $\e(\delta_1,\delta_2)$ the minimal value of $\dim \E(-,-)$ on $\PHom(\delta_1)\times \PHom(\delta_2)$.
We found many analogous results about general representations for general presentations.
For example, the following analogue of Kac's canonical decomposition.
\begin{definition*}[\cite{DF}] A weight vector $\delta\in\mathbb{Z}^{Q_0}$ is called {\em indecomposable} if a general presentation in $\PHom(\delta)$ is indecomposable.
	We call $\delta=\delta_1\oplus \delta_2\oplus\cdots\oplus\delta_s$ the {\em canonical decomposition} of $\delta$ if a general element in $\PHom(\delta)$ decompose into (indecomposable) ones in each $\PHom(\delta_i)$.
\end{definition*}
\begin{theorem}[{\cite[Theorem 4.4]{DF}}] \label{T:introCDPHom} $\delta=\delta_1\oplus \delta_2\oplus\cdots\oplus\delta_s$ is the canonical decomposition of $\delta$ if and only if $\delta_1,\cdots,\delta_s$ are indecomposable, and $\e(\delta_i,\delta_j)=0$ for $i\neq j$.
\end{theorem}

However, an analogue of the following interesting result in \cite{S2} is missing.
\begin{theorem}[{\cite[Theorem 5.4]{S2}}] \label{T:introext} Let $\alpha$ and $\beta$ be dimension vectors for the quiver $Q$. Then 
	$$\ext(\alpha,\beta) = \max_{\beta\twoheadrightarrow \beta'}\{-\innerprod{\alpha,\beta'}\} =  \max_{\alpha'\hookrightarrow \alpha}\{-\innerprod{\alpha',\beta}\}.$$
\end{theorem}
\noindent Here $\ext(\alpha,\beta) := \min\{\dim\Ext(M,N) | M \in \rep_\alpha(Q), N \in \rep_\beta(Q)\}$ and $\innerprod{-,-}$ is the Euler form of $Q$.

%\noindent Note the theorem is equivalent to that $$\hom(\alpha,\beta) = \max_{\beta\twoheadrightarrow \beta'}\{\innerprod{\alpha,\beta'} =  \max_{\alpha'\hookrightarrow \alpha}\{\innerprod{\alpha',\beta}\}.$$

It seems reasonable to find an analogue for $\dim \E(d_1,d_2)$ where $d_1$ and $d_2$ are general presentations.
Recall that the space $\E(d_1,d_2)$ only depends on the cokernel of $d_2$, so it makes sense to define $\E(d,M)$ as $\E(d,d_M)$ where $d_M$ is any presentation of $M$.
In fact, we find something more general, a formula for $\dim \E(d,M)$ where $d$ is a general presentation while $M$ is an {\em arbitrary} representation. However, the statement is not as neat as Schofield's.

To state this result, we need to introduce the other main character of this paper, the tropical $F$-polynomials of representations, which interplay with the general presentations.
Let $M$ be a finite-dimensional representation of $A$. 
\begin{definition*} The {\em tropical $F$-polynomial} $f_M$ of a representation $M$ is the function $(\mb{Z}^{Q_0})^* \to \mb{Z}_{\geq 0}$ defined by
	$$\delta \mapsto \max_{L\hookrightarrow M}{\delta(\dv L)}.$$
	The {\em dual} tropical $F$-polynomial $\fc_M$ of a representation $M$ is the function $(\mb{Z}^{Q_0})^* \to \mb{Z}_{\geq 0}$ defined by
	$$\delta \mapsto \max_{M\twoheadrightarrow N}{\delta(\dv N)}.$$
\end{definition*}

We denote by $\hom(\delta,M)$ and $\e(\delta,M)$ the dimension of the kernel and cokernel of 
$$\Hom(P_+,M) \to \Hom(P_-,M)$$
which is induced from a general presentation $P_-\to P_+$ in $\PHom(\delta)$.
Similarly we can define $\hom(M,\dtc)$ and $\ec(M,\dtc)$ using a general injective presentation of weight $\dtc$.
Here is our analogue of Theorem \ref{T:introext}.
\begin{theorem}[{Theorem \ref{T:HomE}}] \label{T:introHomE} For any representation $M$ and any $\delta\in \mb{Z}^{Q_0}$, there is some $n\in\mb{N}$ such that
		\begin{align*}
		{f}_M(n\delta) &= \hom(n\delta,M), & \fc_M(-n\delta) &= {\e}(n\delta,M).
		\end{align*} 
		Similarly, for any representation $M$ and any $\dtc\in \mb{Z}^{Q_0}$, there is some $\check{n}\in\mb{N}$ such that
		\begin{align*} \fc_M(\check{n}\check{\delta}) &= \hom(M,\check{n}\check{\delta}), & {f}_M(-\check{n}\check{\delta}) &= \ec(M,\check{n}\check{\delta}).\end{align*}
\end{theorem}
\noindent Moreover, $n$ can be replaced by $kn$ for any $k\in\mb{N}$.
	If $m$ is the minimum of all such $n$, then $m\delta$ can not be decomposed as $m\delta = k\delta \oplus (m-k)\delta$ for any $k$. In particular, if $\e(\delta,\delta)=0$, then $m=1$. 
	
When $A$ is the {\em Jacobian algebra} of a {\em quiver with potential}, we show that $m=1$ in the following two cases:
\begin{enumerate}	\item $M$ is {\em negative reachable} (Theorem \ref{T:HomEQP});
	\item The quiver is {\em mutation-acyclic} and $M$ is the cokernel of a general presentation (Corollary \ref{C:HomEQP}).
\end{enumerate}
\noindent The following direct consequence says that the evaluation of $f_M$ computes the asymptotic $\hom(a\delta, M)$ as $a$ increases, which generalizes a result of W. Crawley-Boevey on quivers without relations \cite{CB}.
\begin{corollary}[{Corollary \ref{C:HomE}}] The following limits exist and we have the equalities
\begin{align*} \lim_{a\to \infty}\frac{1}{a}\hom(a\delta, M) &= f_M(\delta) & \lim_{a\to \infty}\frac{1}{a}\e(a\delta, M) &= \fc_M(-\delta); \\
\lim_{a\to \infty}\frac{1}{a}\hom(M, a\dtc)&=\fc_M(\dtc) & \lim_{a\to \infty}\frac{1}{a}\ec(M, a\dtc) &= f_M(-\dtc).
\end{align*}
\end{corollary}

The tropical $F$-polynomial $f_M$ is completed determined by the Newton polytope of $M$.
\begin{definition*}
	The {\em Newton polytope ${\N}(M)$ of a representation} $M$ is the convex hull of
	$$\{ \dv L \mid L\hookrightarrow M \}$$
	in $\mb{R}^{Q_0}$.	The {\em dual} Newton polytope $\check{\N}(M)$ of a representation $M$ is the convex hull of
	$$\{ \dv N \mid M\twoheadrightarrow N \}$$
	in $\mb{R}^{Q_0}$.	
\end{definition*}

As an easy consequence of Theorem \ref{T:introHomE}, we get a presentation of $\N(M)$. 
\begin{theorem}[Theorem \ref{T:Newton}] The Newton polytope $\N(M)$ is defined by 
	$$\{\gamma\in \mb{R}^{Q_0}\mid \delta(\gamma)\leq \hom(\delta,M),\ \forall \delta \in \mb{Z}^{Q_0} \}.$$
	The dual Newton polytope $\check\N(M)$ is defined by 
	$$\{\gamma\in \mb{R}^{Q_0}\mid \check{\delta}(\gamma)\leq \hom(M,\check{\delta}),\ \forall \check{\delta} \in \mb{Z}^{Q_0} \}.$$
\end{theorem}

It is then natural to study the Newton polytope of a representation.
%It turns out that the facets and vertices and their dual - normal vectors and normal cones are fundamental to some open problems posed in \cite{DF}. 
The vertices and facets were studied in detail in \cite{Fc}. In this paper we focus on their duals - the normal vectors and normal cones.
Recall that
\begin{definition*}[\cite{DF}] A weight vector $\delta\in\mathbb{Z}^{Q_0}$ is called {\em real} if $\e(d,d)=0$ for some $d\in\PHom(\delta)$. 
%	is called {\em tame} if it is not real but $\e(\delta,\delta)=0$; is called {\em wild} if $\e(\delta,\delta)>0$.	
A maximal set of real indecomposable weight vectors $\{\delta_1,\dots,\delta_n\}$ such that $\e(\delta_i,\delta_j)=0$ for $i\neq j$ is called a {\em cluster}. 
\end{definition*}
%If an indecomposable $\delta$ is real or tame, then by Theorem \ref{T:introCDPHom}
%the canonical decomposition of $m\delta$ is $m$ copies of $\delta$ for any $m\in\mathbb{N}$. 
%In particular, $\delta$ is indivisible. 

\begin{definition*} For a fixed algebra $A$, a weight vector $\delta$ is called {\em normal} if it is an outer normal vector of the Newton polytope of some $M\in \rep A$.
%	A maximal set of indecomposable normal vectors $\{\delta_1,\dots,\delta_n\}$ such that $\e(\delta_i,\delta_j)=0$ for $i\neq j$ is called a {\em generalized cluster}.
\end{definition*}

\noindent We show in Corollary \ref{C:indnormal} that any indivisible outer normal vector of $\N(M)$ must be indecomposable. It is natural to ask if the converse is true.

\begin{question*} Is any indecomposable $\delta$-vector normal?
\end{question*}
\noindent The answer is positive if $\delta$ is real. We give an equivalent condition for $\delta$ being normal (Proposition \ref{P:normalstable}).

%Here is a question posed in \cite{DF}.
%\begin{question*} If $\delta$ is indecomposable and wild, are all $m\delta$ indecomposable and wild? \end{question*}
%\noindent We show that the answer to this question is positive if $\delta$ is normal (Proposition \ref{P:wild}).

In our setting, the {\em normal cone} $\F_\gamma(M)$ of a vertex $\gamma\in \N(M)$ is the cone spanned by $\delta$ satisfying
\begin{equation*} \delta(\gamma) = f_M(\delta).
\end{equation*}
The two most important normal cones are the ones corresponding to the vertices $0$ and $M$. 
The lattice points inside the cones are precisely
\begin{align*}
\{\delta\in \mb{Z}^{Q_0} &\mid \hom(n\delta,M)=0  \text{ for some $n\in\mb{N}$} \};\\
\{\delta\in \mb{Z}^{Q_0} &\mid {\e}(n\delta,M)=0  \text{ for some $n\in\mb{N}$} \}.
\end{align*}
Clearly $\F_0(M)$ always contains the {\em negative cluster} $(-e_1,\dots,-e_n)$ and $\F_M(M)$ always contains the {\em positive cluster} $(e_1,\dots,e_n)$.
In particular, Theorem \ref{T:introHomE} provide us a presentation for them (Corollary \ref{C:HE0}).

Our most important result about the normal cones is the following
\begin{theorem}[{Theorem \ref{T:Mcones}}] \label{T:introMcones} Let $\bs{\delta}_1, \dots, \bs{\delta}_m$ be finitely many clusters.
	Then there is some representation $M$ such that each $\bs{\delta}_i$ spans a normal cone of $\N(M)$.
\end{theorem}

The normal cones of $\N(M)$ fit together into a complete fan $\F(M)$, the {\em normal fan} of $\N(M)$.
The generalized cluster fan defined below refines the cluster fan introduced in \cite{DF}.
\begin{definition*}
	Let $\F(\rep A)$ be the set of all cones spanned by $\{\delta_1,\dots,\delta_p\}$ such that each $\delta_i$ is normal and $\e(\delta_i,\delta_j)=0$ for $i\neq j$.
	It turns out that $\F(\rep A)$ forms a simplicial fan. We call it {\em generalized cluster fan}.
\end{definition*}

\begin{proposition}[Proposition \ref{P:fanCC}] \label{P:introfanCC} The fan ${\F}(M)$ is a coarsening of the generalized cluster fan $\F(\rep A)$.
\end{proposition}

To study the dual picture, namely the $1$-skeleton of $\N(M)$, 
we need the Schur representations, especially the real ones.
\begin{definition*}	A representation $V$ is called {\em Schur} if $\Hom(V,V)=k$.
	It is called {\em real Schur} if in addition we require $\Ext^1(V,V)=0$.
\end{definition*}

Suppose that $\{\delta_-\}\cup \bs{\delta}_0$ and $\{\delta_+\}\cup \bs{\delta}_0$ are two {\em adjacent} clusters.
We assume that $(\delta_-,\delta_+)$ is a {\em regular exchange pair}, that is, $\e(\delta_-,\delta_+)=1$.
In this case we define the sign of $\delta_-$ in the cluster $\{\delta_-\}\cup \bs{\delta}_0$ to be negative,
and the sign of $\delta_+$ in the cluster $\{\delta_+\}\cup \bs{\delta}_0$ to be positive.
Let $d_-$ and $d_+$ be general presentations of weight $\delta_-$ and $\delta_+$.
Let $L=\Coker(d_+)$ and $N=\Coker(d_-)$, then $\hom(L, \tau N)=\e(d_-,L)=1$.
We consider the exact sequence 
$$0\to K \to L \to \tau N \to C \to 0.$$
Let $I$ be the image of $L \to \tau N$.
It is not hard to show that $I$ is a real Schur representation (Lemma \ref{L:Schur}).

Let $\bs{\delta}=\{\delta_1,\delta_2,\dots,\delta_n\}$ be a cluster, and $\bs{\delta}_j'=(\bs{\delta} \setminus \{\delta_j\})\cup \{\delta_j'\} $ be the {\em adjacent cluster}. 
Let $I_j$ be defined as above for each (unordered) exchange pair $\{\delta_j, \delta_j'\}$,
and $\ep_j$ be the sign of $\delta_j$ in $\bs{\delta}$.
We say $\bs{\delta}$ is a {\em regular cluster} if each exchange pair is regular.
Below we use the upright $\updelta$ to denote the usual delta-function. We write $\delta^\perp$ for the abelian subcategory
$$\delta^\perp := \{M\in \rep A \mid \hom(\delta,M) = \e(\delta,M) =0 \}.$$

\begin{theorem}[Theorem \ref{T:dualcluster}] \label{T:introdualcluster} Let $\{\delta_i\}_{i}$ be a regular cluster and $I_j$ be defined as above. Then 
	$$\hom(\delta_i, I_j) = [\ep_j]_+ \updelta(i,j)\ \text{ and }\ \e(\delta_i, I_j)= [-\ep_j]_+ \updelta(i,j).$$
	Moreover, the simple objects in the category $\delta_I^\perp:=\bigcap_{i\in I} \delta_i^\perp$ are precisely $I_j$ ($j\notin I$).
\end{theorem}

Now we state the results about the $1$-skeleton of $\N(M)$. 
\begin{proposition}[Proposition \ref{P:edgequiver}] \label{P:introedge} If $\br{L_- L_+}$ is an edge in $\N(M)$, then either $L_- \subset L_+$ or $L_+ \subset L_-$.
	Say $L_- \subset L_+$, then 
	$$L_- = t_\dtb(M)\ \text{ and } \ L_+ = \tc_\dtb(M)\ \text{ for any $\delta$ in the interior of $\F_{\br{L_- L_+}}(M)$}.$$
\noindent Here, $t_\dtb$ and $\tc_\dtb$ are two functors introduced in \cite{Fc}, and $\F_{\br{L_- L_+}}(M)=\F_{L_-}(M)\cap \F_{L_+}(M)$.
	Moreover, we have the following \begin{enumerate}
		\item $\delta_+(L_+/L_-)\geq 0$ for any $\delta_+ \in \F_{L_+}(M)$ and $\delta_-(L_+/L_-)\leq 0$ for any $\delta_- \in \F_{L_-}(M)$ with the equality holding only when $\delta_\pm \in \F_{\br{L_- L_+}}(M)$. 
		%\item For any two vertex subrepresentations $L\subset L'$, there is a path from $L$ to $L'$.
		\item If $\F_{L_-}(M)$ is spanned by a regular cluster, then $L_+/L_-$ is a direct sum of isomorphic real Schur representations.
	\end{enumerate}
\end{proposition}

\begin{definition*} We assign the orientation $L_0 \to L_1$ for each edge $\br{L_0 L_1}$ with $L_0\subset L_1$. We call the resulting oriented graph the {\em edge quiver} of $\N(M)$, denoted by $\N_1(M)$.
\end{definition*}

The above results, especially Theorem \ref{T:introMcones}, Proposition \ref{P:introfanCC} and \ref{P:introedge}, when being applied to some special cases, already produce new and non-trivial results.
For example 

\begin{proposition}[Proposition \ref{P:finitetype}] \label{P:introfinitetype} Suppose that $A$ is cluster-finite. Let $M$ be the direct sum of all $\E$-rigid representations. 
	Then the normal fan $\F(M)$ is the cluster fan of $A$, and the edge quiver $\N_1(M)$ is the exchange quiver of $A$. 
\end{proposition}
\noindent  In view of Proposition \ref{P:introfanCC} and \ref{P:introfinitetype}, the generalized cluster fan $\F(\rep A)$ can be viewed heuristically as the normal fan of the infinite dimensional representation $\bigoplus_{M \in \rep A} M$.

\begin{proposition}[Proposition \ref{P:preproj}] Suppose that $A$ is a preprojective algebra of Dynkin type. The vertices of $\N(A)$ are labelled by the ideals $I_w$, and $\F_{I_w}(A)$ is the cluster corresponding to $I_w$. 	
	So $\F(A)$ is the cluster fan $\F(\rep A)$, which is a Weyl fan.
\end{proposition}

Finally let us come back to the generic setting as in Schofield's original paper.
We are interested in determining the Newton polytopes of general representations. 
\begin{theorem}[Theorem \ref{T:genericN}] Let $\alpha$ be any dimension vector of $Q$.
	Each normal cone $\F_\gamma(\alpha)$ of $\N(\alpha)$ contains a cluster.
	Hence the Newton polytope $\N(\alpha)$ is completely determined by Newton polytopes of real Schur representations.
\end{theorem}
\noindent ``Determine" here means that we can explicitly compute all vertices of $\N(\alpha)$ by what we observed in \ref{O:vertex}.
More generally, we are interested in the Newton polytope of the cokernel of a general presentation, especially for the Jacobian algebras.
In some optimistic situation (e.g., Question \ref{q:realincone} is positive), the method would work for such generic Newton polytopes (see Observation \ref{O:real2FG}).
We will explain below why this is an important problem in the cluster algebra theory.
This approach also gives an alternative proof of Schofield's Theorem \ref{T:introext}.

\subsection*{Motivation and Relation to Cluster Algebras}
The tropical $F$-polynomials and general presentations discussed in this paper are originated from the theory of cluster algebras \cite{FZ1}.
We know from \cite{FZ4} that for cluster algebras of geometric type any {\em cluster variable} can be written as 
\begin{equation} \label{eq:cv} X(\dtc) := \b{x}^{-\dtc} F_{\dtc}(\b{y}), \end{equation}
where $\b{y}$ is a certain monomial change of the initial cluster variables $\b{x}$.
Here we use $\b{x}^a$ to denote the monomial $\prod_{i} x_i^{a(i)}$.

If the cluster algebra is {\em skew-symmetric}, we have a {\em nondegenerate quiver with potential} $(Q,\mc{P})$ to model this algebra \cite{DWZ2}. Let $A$ be the {\em Jacobian algebra} associated to $(Q,\mc{P})$.
The above polynomial $F_{\dtc}$ is the {\em $F$-polynomial} of some $\E$-rigid representation $M$ of $A$ \cite{DWZ2}.
Moreover, the minimal injective presentation of $M$ has weight exactly $\dtc$.
Since the coefficients of $F$ are all positive, we can tropicalize it in the usual sense.
The tropicalization is precisely the tropical $F$-polynomial of $M$.

Moreover if $\{X(\dtc_1),\dots,X(\dtc_n)\}$ forms a {\em cluster} in the cluster algebra $\mc{C}(Q)$ then $\{\dtc_1,\dots,\dtc_n\}$ is a cluster in $\rep A$ \cite{DWZ2,DF}.
So the cluster fan $\F^r(\rep A)$ is the original cluster fan for $\mc{C}(Q)$.
In this setting the signed dimension vector $\ep_j \dv I_j$ of $I_j$ in Theorem \ref{T:introdualcluster} is the corresponding {\em $c$-vectors} of the cluster.

If the Jacobian algebra is cluster-finite, then we get an easy consequence of Proposition \ref{P:introfinitetype}.
In this case the Newton polytope is the so-called {\em generalized associahedron} \cite{FZy}.
\begin{corollary}[Corollary \ref{C:QPfinite}] Suppose that $A$ is a cluster-finite Jacobian algebra. Let $M$ be the direct sum of all $\E$-rigid representations of $A$.
	Then the dual fan $\F(M)$ is the cluster fan of $\mc{C}(Q)$, and the edge quiver $\N_1(M)$ is the exchange quiver of $\mc{C}(Q)$. 
	Moreover, the signed dimension vectors of the real Schur representations attached to the arrows from/to a fixed vertex $L$ are the signed $c$-vectors dual to the cluster $\F_L(M)$.
\end{corollary}

The formula \eqref{eq:cv} has a naive generalization where we consider the $F$-polynomial $F_{\dtc}$ of the kernel of a general injective presentation of {\em any} weight $\dtc\in \mb{Z}^{Q_0}$ \cite{D}.
In many cases they are turned out be a basis of the upper cluster algebra $\br{\mc{C}}(Q)$ \cite{P}.
The Newton polytope of this $F_\dtc$ is exactly given by the generic Newton polytope $\N(\dtc)$.

In the meanwhile, a remarkable {\em positive} basis consisting of {\em theta functions} for all cluster algebras was introduced in \cite{GHKK}. 
For each $\dtc$-vector, there is a theta function $\vartheta_\dtc$, 
which is of the form $$\vartheta_\dtc = \b{x}^{-\dtc} \varphi_{\dtc}(\b{y}).$$
In general, the theta function can be a Laurent series, but let us assume it is a Laurent polynomial
so $\varphi_{\dtc}$ is a polynomial with {positive} coefficients.
Another very interesting positive (quantum) basis called {\em triangular basis} was introduced in \cite{Q}. It has a similar form 
$$T_{\dtc,q} = \b{x}^{-\dtc} \mc{F}_{\dtc,q}(\b{y}).$$
In particular, $\varphi_{\dtc}$ and $\mc{F}_{\dtc,q}$ can be tropicalized and the tropicalization is determined by its Newton polytope.	
We have the following conjecture
\begin{conjecture} The Newton polytopes of $\varphi_\dtc$ and $\mc{F}_{\dtc,q}$ are the same as the generic Newton polytope $\N(\dtc)$.
\end{conjecture} 

%Roughly speaking, the {\em $F$-polynomial} of a quiver representation $M$ is the generating series of the topological Euler characteristic of the {\em representation Grassmannian} of $M$:
%$$F_M(\b{y}) = \sum_{\gamma} \chi(\Gr_\gamma(M)) \b{y}^\gamma.$$
%

Another related problem is the Fock-Goncharov duality conjecture \cite[Conjecture 4.1]{FGc}.
Recall that a skew-symmetrizable matrix $B$ gives rise to a pair of cluster varieties $(\mc{A},\mc{X})$, 
and their {\em Langlands dual} $(\mc{A}^\vee,\mc{X}^\vee)$.
The conjecture says that the tropical points $\mc{X}^\vee(\mb{Z}^t)$ of $\mc{X}^\vee$ parametrize a basis of ring of regular functions $\mc{O}(\mc{A})$ of $\mc{A}$, and we can interchange the roles of $\mc{A}$ and $\mc{X}$.
The duality conjecture fails in general, but can hold with some mild assumption, or if replaced with a certain formal version (see \cite{GHKK} for detail). 
Let us assume the parametrizations exist and we denote them by 
$$I_{\mc{A}}: \mc{A}(\mb{Z}^t) \hookrightarrow \mc{O}(\mc{X}^\vee)\ \text{ and }\ I_{\mc{X}^\vee}: \mc{X}^\vee(\mb{Z}^t) \hookrightarrow{} \mc{O}(\mc{A}).$$
The duality conjecture further asserts that we can require the parametrized bases to be {\em universally positive} and satisfy several interesting properties.
One of them concerns the pairing $$\mc{A}(\mb{Z}^t) \times \mc{X}^\vee(\mb{Z}^t) \to \mb{Z} .$$
There are two canonical (conjecturally equal) ways to define this pairing:
\begin{align*} &&  I_{\mc{A}}(a)^{\op{trop}} (x)  && I_{\mc{X}^\vee}(x)^{\op{trop}} (a) &&  \text{for } a\in \mc{A}(\mb{Z}^t),\  x\in \mc{X}^\vee(\mb{Z}^t).
\end{align*}
We give a representation-theoretic interpretation of the pairing in some special cases.
\begin{theorem}[Fock-Goncharov duality pairing] Suppose that $B$ is skew-symmetric.
	The pairings $\mc{A}(\mb{Z}^t) \times \mc{X}^\vee(\mb{Z}^t) \to  \mb{Z}$ given by $I_{\mc{A}}(a)^{\wtd{\op{trop}}}(\dtc)$ and $I_{\mc{X}^\vee}(\dtc)^{\wtd{\op{trop}}}(a)$ are both equal to 
	$\hom(aB^{\T},\dtc) - a\cdot \dtc$ in the following two situations
	\begin{enumerate}
		\item The quiver of $B$ is mutation-equivalent to an acyclic quiver;
		\item Either $I_{\mc{X}^\vee}(\dtc)$ or $I_{\mc{X}}(aB^{\T})$ is a cluster variable, or equivalently either $\dtc$ or $aB^{\T}$ is negative reachable.
	\end{enumerate}
\end{theorem}
\noindent Although the main part of Fock-Goncharov duality conjecture was intensively studied, the meaning of the duality pairing is only known in few cases (e.g., \cite[Proposition 12.1]{FG}). The verification of the equality in this generality is new.

\subsection*{Organization} In Section \ref{S:GP} we review the theory of general presentations developed in \cite{DF}.
In Section \ref{S:Ftrop} we introduce the tropical $F$-polynomial of a representation and its Newton polytope.
We prove our first main result -- Theorem \ref{T:HomE}.
Then we improve this result in the case of quivers with potentials (Theorem \ref{T:HomEQP} and Corollary \ref{C:HomEQP}). 
In Section \ref{S:tf} we review the two pairs of functors considered in \cite{Fc}.
In Section \ref{S:Newton} we give a presentation of the Newton polytope -- Theorem \ref{T:Newton}.
We study the normal vectors and the normal cones, and prove another main result -- Theorem \ref{T:Mcones}.
In Section \ref{S:generic} we propose an algorithm to determine the generic Newton polytopes.
We show in Theorem \ref{T:genericN} that the algorithm works for path algebras.
Observation \ref{O:real2FG} explains why we speculate the algorithm should work more generally.
We make some connection to the cluster algebra theory, including an interpretation of the Fock-Goncharov duality pairing (Theorem \ref{T:FGpairing}). 
In Section \ref{S:dual} we give an explicit construction of dual clusters consisting of real Schur representations in Theorem \ref{T:dualcluster}.
In Section \ref{S:fan&edge} we study the normal fan and edge quiver of the Newton polytope. 
For the general case the two main results here are Propositions \ref{P:fanCC} and \ref{P:edgequiver}.
For the quiver case we state an interesting bijection in Conjecture \ref{T:Schurseq}. 
In Section \ref{S:example} we apply the above results to two special cases. One is the cluster-finite algebra (Proposition \ref{P:finitetype}) and the other is the preprojective algebra of Dynkin type (Proposition \ref{P:preproj}).

\subsection*{Notation and Conventions}
Throughout we only deal with finite-dimensional basic algebras over an algebraically closed field $k$ of characteristic $0$. 
If we write an algebra $A=kQ$, we assume implicitly that $Q$ is finite and has no oriented cycles.
For general $A=kQ/I$, we allow $Q$ to have oriented cycles.
%Although the paper is written in this generality, some of the results are only proved for path algebras.
%Sometimes instead of switching between $A=kQ/I$ and $A=kQ$ we may just say that assume $A$ has no relations.
We denote by $Q_0$ the set of vertices of $Q$.

Unless otherwise stated, unadorned $\Hom$ and other functors are all over the algebra $A$, and the superscript $*$ is the trivial dual for vector spaces.
For direct sum of $n$ copies of $M$, we write $nM$ instead of the traditional $M^{\oplus n}$.
We write $\hom,\ext$ and $\e$ for $\dim\Hom, \dim\Ext$, and $\dim \E$.
%When dealing the hereditary algebras, we write $\Ext$ instead of $\Ext^1$.

%Any rational ray or normal vector will be represented by an indivisible integral vector, that is, an integral vector with no common divisors.
\begin{align*}
& \rep A && \text{the category of finite-dimensional representations of $A$} &\\
& \rep_\alpha(A) && \text{the space of $\alpha$-dimensional representations of $A$} &\\
& S_i && \text{the simple representation supported on the vertex $i$} &\\
& P_i && \text{the projective cover of $S_i$} &\\
& I_i && \text{the injective envelope of $S_i$} &\\
& \dv M && \text{the dimension vector of $M$} & 
\end{align*}

\section{Review on General Presentations} \label{S:GP}
\subsection{The \texorpdfstring{$\E$}{E}-invariant of Presentations} 
Let $A$ be a finite-dimensional basic algebra over an algebraically closed field $k$ of characteristic $0$.
Then $A$ can be presented as a path algebra modulo an ideal generated by {\em admissible} relations: $A=kQ/I$ \cite{ASS}.
We denote by $P_v$ (resp. $I_v$) the indecomposable projective (resp. injective) representation of $A$ corresponding the vertex $v$ of $Q$.
For $\beta \in \mb{Z}_{\geq 0}^{Q_0}$ we write $P(\beta)$ for $\bigoplus_{v\in Q_0} \beta(v)P_v$;
similarly write $I(\beta)$ for $\bigoplus_{v\in Q_0} \beta(v)I_v$.
Following \cite{DF} we call a homomorphism between two projective representations, a {\em projective presentation} (or presentation in short). 
As a full subcategory of the category of complexes in $\rep A$, the category of projective presentations is Krull-Schmidt as
well.

\begin{definition}\footnote{The $\delta$-vector is the same one defined in \cite{DF}, but is the negative of the $\g$-vector defined in \cite{DWZ2}. }
	The {\em $\delta$-vector} (or {\em reduced weight vector}) of a presentation 
	$$d: P(\beta_-)\to P(\beta_+)$$
	is the difference $\beta_+-\beta_- \in \mb{Z}^{Q_0}$.
	When working with injective presentations 
	$$\dc: I(\check{\beta}_+)\to I(\check{\beta}_-),$$
	we call the vector $\check{\beta}_+ - \check{\beta}_-$ the {\em $\check{\delta}$-vector} (or {reduced weight vector}) of $\dc$.
\end{definition}
\noindent Let $\nu$ be the Nakayama functor $\Hom(-,A)^*$.
There is a map still denoted by $\nu$ sending a projective presentation to an injective one
$$P_-\to P_+\ \mapsto\ \nu(P_-) \to \nu(P_+).$$
We say a presentation $d$ {\em nonnegative} if $d$ has no direct summands of form $P_-\to 0$.
If $d$ is nonnegative, then $\Ker(\nu d) = \tau\Coker(d)$ where $\tau$ is the classical {\em Auslander-Reiten translation} \cite{ASS}.

\begin{definition}[{\cite{DWZ2,DF}}] Given any projective presentation $d: P_-\to P_+$, we define $\Hom(d,M)$ and $\E(d,M)$ to be the kernel and cokernel of the induced map:
	\begin{equation} \label{eq:longexact} 0\to \Hom(d,M)\to \Hom(P_+,M) \xrightarrow{C(d,M)} \Hom(P_-,M) \to \E(d, M)\to 0.
	\end{equation}
	Similarly for an injective presentation $\dc: I_+\to I_-$, we define $\Hom(M,\dc)$ and $\Ec(M,\dc)$ to be the kernel and cokernel of the induced map $\Hom(M,I_+) \xrightarrow{} \Hom(M,I_-)$.
	It is clear that 
	$$\Hom(d,M) = \Hom(\Coker(d),M)\ \text{ and }\ \Hom(M,\dc) = \Hom(M,\Ker(\dc)).$$
\end{definition}	
\noindent In this paper we never use $\Hom(d,M)$ for the above $k$-linear map $C(d,M)$ induced by $d$.

\begin{lemma}[{\cite{DF}}] \label{L:E} We have the following properties \begin{enumerate}
		\item	Any exact sequence $0\to L \to M \to N \to 0$ in $\rep A$ gives the long exact sequence:
		$$0\to \Hom(d,L)\to \Hom(d,M) \to \Hom(d,N) \to \E(d, L) \to \E(d, M) \to \E(d, N)\to 0.$$
		\item	$\E(d,M) \supseteq \Ext^1(\Coker(d),M)$ for any $d$ and $M$.
		\item	$\E(d,M) \cong \Hom(M, \Ker(\nu d))^*$ for any $d$ and $M$.
	\end{enumerate}	
\end{lemma}
\noindent Readers can easily formulate the analogous statements for $\Ec$. 

Sometimes it is convenient to view presentations as elements in the homotopy category $K^b(\proj A)$ of bounded complexes of projective representations of $A$.
Our convention is that $P_-$ sits in degree $-1$ and $P_+$ sits in degree $0$.
Then the $\delta$-vector of a presentation is just the corresponding element in the Grothendieck group of $K^b(\proj A)$.
Given any two presentation $d_1$ and $d_2$, we also define 
$$\E(d_1,d_2) = \Hom_{K^b(\proj A)} (d_1,d_2[1]).$$
It turns out (\cite{DF}) that
	$$\E(d_1,d_2) = \E(d_1,\Coker(d_2))\ \text{ and }\ \Ec(\dc_1,\dc_2) = \Ec(\Ker(\dc_1),\dc_2).$$

For any representation $M$, we denote by $d_M$ (resp. $\dc_M$) its minimal projective (resp. injective) presentation.
Given any two representation $M$ and $N$, we define 
$$\E(M,N):= \E(d_M,N)\ \text{ and }\ \Ec(M,N):= \Ec(M,\dc_N).$$
%	The vector $\delta_M$ (resp. $\check{\delta}_M$) of a representation $M$ of $A$ is the $\delta$-vector (resp. $\check{\delta}$-vector) of its minimal projective (resp. injective) presentation.
\subsection{General Presentations}
By a {\em general presentation} in $\Hom(P_-,P_+)$, we mean a presentation in some open (and thus dense) subset of $\Hom(P_-,P_+)$.
Any $\delta\in \mb{Z}^{Q_0}$ can be written as $\delta = \delta_+ - \delta_-$ where $\delta_+=\max(\delta,0)$ and $\delta_- = \max(-\delta,0)$. We put 
$$\PHom(\delta):=\Hom(P(\delta_-),P(\delta_+)).$$
It is well known that a general presentation in $\Hom(P(\beta_-),P(\beta_+))$ is homotopy equivalent to a general presentation in $\PHom(\beta_+-\beta_-)$ for any $\beta_-,\beta_+ \in\mb{Z}_{\geq 0}^{Q_0}$. 

There is some open subset $U$ of $\PHom(\delta)$ such that for any $d\in U$ we have
\begin{enumerate}
	\item $\Hom(d,M)$ has constant dimension for a fixed $M\in \rep A$.
	\item $\Coker(d)$ has constant subrepresentation lattice.
\end{enumerate}
Note that (1) implies that $\E(d,M)$ has constant dimension as well. 
It follows from (1) or (2) that $\Coker(d)$ has a constant dimension vector $\alpha$.
In fact, we can ask $\Coker(d)$ lie in a fixed irreducible component of $\rep_\alpha(A)$ (see \cite[Section 2]{DF}). 
We denote by $\Coker(\delta)$ the cokernel of a general presentation in $\PHom(\delta)$.
Similarly we can define the injective presentation space $\IHom(\dtc)$, 
and denote by $\Ker(\dtc)$ the kernel of a general element there.
\begin{definition} \label{D:home} We denote by $\hom(\delta,M)$ and $\e(\delta,M)$ the value of 
	$\hom(d,M)$ and $\e(d,M)$ for a general presentation $d\in \PHom(\delta)$.
	$\hom(M,\dtc)$ and $\ec(M,\dtc)$ are defined analogously.
\end{definition}
\noindent Recall the isomorphism $\Hom(P_i,P_j) \cong \Hom(I_i, I_j) = \Hom(\nu P_i, \nu P_j)$.
If $d$ is general in $\PHom(\delta)$, then $\nu d$ is general in $\IHom(-\delta)$.
We obtain the obvious relations  
\begin{equation}\label{eq:hedual} \hom(\delta,M)= \ec(M,-\delta)\ \text{ and }\ \e(\delta,M)= \hom(M,-\delta).
\end{equation}  

\begin{definition}[\cite{DF}] A weight vector $\delta\in\mathbb{Z}^{Q_0}$ is called {\em indecomposable} if a general presentation in $\PHom(\delta)$ is indecomposable. We call $\delta=\bigoplus_{i=1}^s \delta_i$ a {\em decomposition} of $\delta$ if a general element $d$ in $\PHom(\delta)$ decompose into $\bigoplus_{i=1}^s d_i$ with each $d_i\in \PHom(\delta_i)$.
It is called the {\em canonical decomposition} of $\delta$ if each $d_i$ is indecomposable.
\end{definition}

The function $\dim \E(-,-)$ is upper semi-continuous on $\PHom(\delta_1)\times \PHom(\delta_2)$.
We denote by $\e(\delta_1,\delta_2)$ the minimal value of $\dim \E(-,-)$ on $\PHom(\delta_1)\times \PHom(\delta_2)$.
One of the motivation of introducing the space $\E$ is the following theorem.
\begin{theorem}[{\cite[Theorem 4.4]{DF}}] \label{T:CDPHom} $\delta=\delta_1\oplus \delta_2\oplus\cdots\oplus\delta_s$ is the canonical decomposition of $\delta$ if and only if $\delta_1,\cdots,\delta_s$ are indecomposable, and $\e(\delta_i,\delta_j)=0$ for $i\neq j$.
\end{theorem}

\begin{definition}[\cite{DF}] A weight vector $\delta\in\mathbb{Z}^{Q_0}$ is called {\em real} if $\e(d,d)=0$ for some $d\in\PHom(\delta)$; is called {\em tame} if it is not real but $\e(\delta,\delta)=0$; is called {\em wild} if $\e(\delta,\delta)>0$.
\end{definition}

\noindent If an indecomposable $\delta$ is real or tame, then by Theorem \ref{T:CDPHom}
the canonical decomposition of $m\delta$ is a sum of $m$ copies of $\delta$ for any $m\in\mathbb{N}$. In particular, $\delta$ is indivisible. %Here is a question posed in \cite{DF}.
%\begin{question} \label{q:wild} If $\delta$ is indecomposable and wild, are all $m\delta$ indecomposable and wild? \end{question}

\subsection{\texorpdfstring{$\E$}{E}-rigid Presentations}
The group $\Aut_A(P_-)\times\Aut_A(P_+)$ acts on $\Hom(P_-,P_+)$ by 
$(g_-,g_+)d =g_+ d g_-^{-1}$.
The space $\E(d,d)$ can be interpreted as the normal space to the orbit of $d$ in $\Hom(P_-, P_+)$.
\begin{definition}
A presentation $d$ is called {\em rigid} \footnote{In \cite{KV} this is called {\em presilting}, which is defined for any complex in $K^b(\proj A)$.} if $\E(d,d)=0$ ($\Ec(\dc,\dc)=0$ for an injective presentation $\dc$).
A representation $M$ is called {\em $\E$-rigid} \footnote{Due to the equation $\E(M,M)=\Hom(M,\tau M)^*$, it is also called {\em $\tau$-rigid} in \cite{AIR}.} (resp. $\Ec$-rigid) if $\E(M,M)=0$ (resp. $\Ec(M,M)=0$).
%It is called $\E$-birigid if it is both $\E$-rigid and $\Ec$-rigid.
\end{definition}
\noindent So the orbit of such a presentation is dense in its ambient space.
In this case the weight vector of $d$ must be real.
The dual of Lemma \ref{L:E}.(3) says that $\Ec(M,\dc) \cong \Hom(\Coker(\nu^{-1} \dc) ,M)^*$.
So we have that 
\begin{equation}\label{eq:eec} \E(d,d) \cong \Hom(\Coker(d),\Ker(\nu d))^* \cong \Ec(\nu d, \nu d). 
\end{equation}
This implies that $d$ is rigid if and only if $\nu d$ is rigid.
%So if $M$ is $\E$-rigid, then $\tau E$ is $\Ec$-rigid.

One can always complete a rigid presentation $d$ to a {\em maximal rigid} one $\tilde{d}$, in the sense that 
$\E(\tilde{d}\oplus d', \tilde{d}\oplus d') \neq 0$ for any indecomposable $d'\notin \ind(d)$.
Here we denote by $\ind(d)$ the set of nonisomorphic indecomposable direct summands of $d$.
The maximal rigid presentation can be characterized as follows.
\begin{theorem}[{\cite[Theorem 5.4]{DF}}, \cite{AIR}] \label{T:maxrigid} The following are equivalent for a rigid presentation $d$. \begin{enumerate}
\item	$d$ is maximal rigid; 
\item	$|\ind(d)|=|Q_0|$; 
\item 	$\ind(d)$ generates $K^b(\proj A)$. 
\end{enumerate}
\end{theorem}

\begin{definition}
If $d$ is maximal rigid, then we call $\ind(d)$ a {\em cluster} of presentations.
We also call the weight vectors of presentations in $\ind(d)$ a {\em cluster} of $\delta$-vectors.
A maximal set of indecomposable presentations $\{d_1,\dots,d_r\}$ satisfying $\e(d_i,d_j)= 0$ for $i\neq j$ is called a {\em generalized cluster} of presentations. 
A maximal set of indecomposable weight vectors $\{\delta_1,\dots,\delta_r\}$ satisfying $\e(\delta_i,\delta_j)= 0$ for $i\neq j$ is called a {\em generalized cluster} of $\delta$-vectors.
%Their weights $\{\delta_1,\dots,\delta_r\}$ is also called a {\em generalized cluster} of $\delta$-vectors.
\end{definition}

\begin{proposition}[{\cite[Proposition 5.7]{DF}}] \label{P:+-} If a rigid presentation $d$ is {almost complete}, that is, $|\ind(d)|=|Q_0|-1$,
	then it has exactly two complements $d_-$ and $d_+$. They are related by the triangle $d_+\to d'\to d_-^e\to d_+[1]$ and $d_+^e\to d''\to d_-\to d_+^e[1]$, where $e=\dim\E(d_-,d_+)$. Moreover, both $d'\oplus d_-$ and $d''\oplus d_+$ are rigid and $\E(d_+,d_-)=\E(d_+,d')=\E(d'',d_-)=0$. In particular, $e=1$ if and only if  $d'=d''$ belongs to the subcategory generated by $\ind(d)$.
\end{proposition}

\begin{definition} We call the above pair $(d_-,d_+)$ an {\em exchange pair} of presentations.
	If $e=1$, the exchange pair is called {\em regular}.
	The two clusters $\{d_-\}\cup \ind(d)$ and $\{d_+\}\cup \ind(d)$ are called {\em adjacent} to each other.
	A cluster $\{d_1,\dots,d_n\}$ is called {\em regular} if each $\{d_i, d_i'\}$ can be ordered to be a regular exchange pair, where $d_i'$ appears in the adjacent cluster $(d_1,\dots,d_i',\dots,d_n)$.
\end{definition}

\noindent An open problem posed in \cite{DF} is how to characterize algebras for which all clusters are regular.
% It is known that the Jacobian algebra of a quiver with generic potential is such an algebra.

\section{Tropical \texorpdfstring{$F$}{F}-polynomials and General Presentations} \label{S:Ftrop}
\subsection{Tropical \texorpdfstring{$F$}{F}-polynomials} \label{ss:Ftrop}
We keep assuming that $A=kQ/I$.
Throughout we identify the Grothendieck group $K_0(\rep A)$ with $\mb{Z}^{Q_0}$.
Let $M$ be a finite-dimensional representation of $A$. 

\begin{definition}\label{D:tropf} The {\em tropical $F$-polynomial} $f_M$ of a representation $M$ is the function $(\mb{Z}^{Q_0})^* \to \mb{Z}_{\geq 0}$ defined by
	$$\delta \mapsto \max_{L\hookrightarrow M}{\delta(\dv L)}.$$
	The {\em dual} tropical $F$-polynomial $\fc_M$ of a representation $M$ is the function $(\mb{Z}^{Q_0})^* \to \mb{Z}_{\geq 0}$ defined by
	$$\delta \mapsto \max_{M\twoheadrightarrow N}{\delta(\dv N)}.$$
\end{definition}
\noindent Clearly $f_M$ and $\fc_M$ are related by $f_M(\delta)-\fc_M(-\delta)= \delta(\dv M)$.
The definition of $f_M$ is motivated by the $F$-polynomial of $M$ defined in \cite{DWZ2}
	$$F_M(\b{y}) = \sum_{\gamma} \chi(\Gr_\gamma(M)) \b{y}^\gamma,$$
where $\Gr_\gamma(M)$ is the variety parametrizing the $\gamma$-dimensional subrepresentations of $M$, and $\chi(-)$ is the topological Euler characteristic.
In general $\chi(\Gr_\gamma(M))$ may not be a positive number. If the $F$-polynomial $F_M$ has non-negative coefficients, then the tropical $F$-polynomial $f_M$ is the usual tropicalization of $F_M$.

\begin{definition}
	The {\em Newton polytope ${\N}(M)$ of a representation} $M$ is the convex hull of
	$$\{ \dv L \mid L\hookrightarrow M \}$$
	in $\mb{R}^{Q_0}$.	The {\em dual} Newton polytope $\check{\N}(M)$ of a representation $M$ is the convex hull of
	$$\{ \dv N \mid M\twoheadrightarrow N \}$$
	in $\mb{R}^{Q_0}$.	
\end{definition}

\begin{remark} We have two remarks. \begin{enumerate} 
	\item The tropical $F$-polynomial $f_M$ is completed determined by the Newton polytope of $M$.
	\item It is shown in \cite{Fc} that the Newton polytope of $M$ is the same as the usual Newton polytope of the polynomial $F_M$.
	\end{enumerate}
\end{remark}

\begin{lemma}[{\cite[Proposition 3.2]{DWZ2}}] \label{L:directsum}  $F_{M\oplus N}= F_MF_N$ for any two representations $M$ and $N$.
	In particular, we have that $f_{M\oplus N} = f_{M} + f_{N}$.
\end{lemma}

When paired with a dimension vector or evaluated by some $f_M$, 
a weight $\delta$ is viewed as an element in $(\mb{Z}^{Q_0})^*$ via the usual dot product.
It follows from \eqref{eq:longexact} that for any presentation $d$ of weight $\delta$,
\begin{align} \label{eq:heform} \delta(\dv M) &= \hom(d,M) - \e(d,M);\\
\label{eq:heformdual} \check{\delta} (\dv M) &= \hom(M,\dc) - \ec(M,\dc).
\end{align}
Let $M\to N$ be a homomorphism. We fix some general presentation $d$ of weight $\delta$.
Throughout we use the notation $\Hom(\delta,M)\to \Hom(\delta,N)$ for the induced map $\Hom(d,M)\to \Hom(d,N)$.
The notation $\E(\delta,M)\to \E(\delta,N)$ has the similar meaning.

\begin{lemma} \label{L:HomE} We have the following inequalities for any representation $M$ and any $\delta\in \mb{Z}^{Q_0}$
	\begin{align*}
	{f}_M(\delta) &\leq \hom(\delta,M), & \fc_M(-\delta) &\leq {\e}(\delta,M);\\
	\fc_M(\check{\delta}) &\leq \hom(M,\check{\delta}), & {f}_M(-\check{\delta}) &\leq \ec(M,\check{\delta}).
	\end{align*}
\end{lemma}
\begin{proof} Since $\Hom(\delta,L) \hookrightarrow \Hom(\delta,M)$ for any subrepresentation $L$ of $M$, 
	we have that $\delta(\dv L) \leq \hom(\delta,L) \leq \hom(\delta,M)$. Hence ${f}_M(\delta) \leq \hom(\delta,M)$.
	Then $\fc_M(-\delta) \leq {\e}(\delta,M)$ follows from \eqref{eq:heform}.
	The other half is proved similarly.
\end{proof}

Here is the main result of this section.
\begin{theorem} \label{T:HomE} For any representation $M$ and any $\delta\in \mb{Z}^{Q_0}$, there is some $n\in\mb{N}$ such that
	\begin{align*}
	{f}_M(n\delta) &= \hom(n\delta,M), & \fc_M(-n\delta) &= {\e}(n\delta,M).
	\end{align*}
Similarly, for any representation $M$ and any $\dtc\in \mb{Z}^{Q_0}$, there is some $\check{n}\in\mb{N}$ such that
	\begin{align*} \fc_M(\check{n}\check{\delta}) &= \hom(M,\check{n}\check{\delta}), & {f}_M(-\check{n}\check{\delta}) &= \ec(M,\check{n}\check{\delta}).\end{align*}
Moreover, $n$ can be replaced by $kn$ for any $k\in\mb{N}$.
If $m$ is the minimum of all such $n$, then $m\delta$ can not be decomposed as $m\delta = k\delta \oplus (m-k)\delta$ for any $k$.
In particular, if $\delta$ is not wild, then $m=1$.  
\end{theorem}

\begin{example}\label{ex:n>1} We remark that $n$ or $\check{n}$ may not always chosen to be $1$.
	Let $Q$ be the three-arrow Kronecker quiver $\Kronthree{\bullet}{\bullet}$. 
	Consider $\delta=(1,-1)$ and $M\in \rep_{(3,3)}(Q)$ given by 
	$$M(a)=\sm{0&0&-1\\ 0&0&0\\ 1&0&0},\quad M(b)=\sm{0&1&0\\ -1&0&0\\ 0&0&0},\quad M(c)=\sm{0&0&0\\ 0&0&1\\ 0&-1&0}.$$
Then one can easily check that $f_M(\delta)=\hom(n\delta,M)=0$ for any $n\geq 2$ but $\hom(\delta,M)=1$.
\end{example}

Before giving a proof, we hasten to mention an interesting corollary (Corollary \ref{C:HomE}). It says that the evaluation of $f_M$ is related to the asymptotic  $\hom(a\delta, M)$ as $a$ increases, generalizing a result of W. Crawley-Boevey on quivers without relations \cite{CB}.

\begin{lemma} \label{L:homineq} If $\delta=\delta_1+\delta_2$, then $\hom(\delta,M) \leq \hom(\delta_1,M)+\hom(\delta_2,M)$ for any $M$.
	If $\delta=\delta_1\oplus \delta_2$, then $\hom(\delta,M) = \hom(\delta_1,M)+\hom(\delta_2,M)$ for any $M$.
	Moreover, all $\hom$ can be replaced by $\e$.	
\end{lemma}
\begin{proof} Let $d_i$ be a general presentation of weight $\delta_i$ ($i=1,2$).
	Then the weight of $d_1\oplus d_2$ is $\delta$ and $\hom(d_1\oplus d_2,M) = \hom(d_1,M) + \hom(d_2,M) = \hom(\delta_1,M) + \hom(\delta_2,M)$.	
	By the lower semi-continuity, we have that $\hom(\delta,M) \leq  \hom(d_1\oplus d_2,M)$.	
	If $\delta=\delta_1\oplus \delta_2$, then we can assume $d_1\oplus d_2$ is general, so $\hom(\delta,M) = \hom(\delta_1,M)+\hom(\delta_2,M)$.
\end{proof}

\begin{question} Is it true that if $\hom(\delta,M) = \hom(\delta_1,M)+\hom(\delta_2,M)$ for any $M$, then $\delta=\delta_1\oplus \delta_2$?
\end{question}

\begin{corollary} \label{C:HomE} The following limits exist and we have the equalities
	\begin{align*} \lim_{a\to \infty}\frac{1}{a}\hom(a\delta, M), &= f_M(\delta) & \lim_{a\to \infty}\frac{1}{a}\e(a\delta, M) &= \fc_M(-\delta); \\
		\lim_{a\to \infty}\frac{1}{a}\hom(M, a\dtc)&=\fc_M(\dtc), & \lim_{a\to \infty}\frac{1}{a}\ec(M, a\dtc) &= f_M(-\dtc).
	\end{align*}
\end{corollary}
\begin{proof} Let $n$ be the number such that $f_M(n\delta) = \hom(n\delta,M)$ as in Theorem \ref{T:HomE}.
	For any $a\in\mb{N}$ we can write $a=qn+r$ with $q,r\in\mb{Z}_{\geq 0}$ and $r<n$.
Then by Lemma \ref{L:homineq} \begin{align*} \frac{1}{a} \hom(a\delta, M) &\leq \frac{1}{qn}(q\hom(n\delta,M)+r\hom(\delta,M)) \\
	&\leq f_M(\delta) + \frac{1}{q}\hom(\delta,M),
\end{align*}
which tends to $f_M(\delta)$ as $a\to \infty$.
On the other hand, we have by Lemma \ref{L:HomE} that 
$$\frac{1}{a}\hom(a\delta,M) \geq \frac{1}{a}f_M(a\delta) = f_M(\delta).$$
Hence the limit exists and equals to $f_M(\delta)$. The others can be proved similarly.
\end{proof}

The proof of Theorem \ref{T:HomE} requires some preparation.

\subsection{Stability and Semi-invariants}
A. King introduced Mumford's GIT into the setting of quiver representation theory \cite{Ki}.
Recall that any weight $\delta\in \mb{Z}^{Q_0}$ gives a multiplicative {\em characters} $\chi_\delta$ of $\GL_\alpha:=\prod_{v\in Q_0} \GL_{\alpha(v)}$:
\begin{equation*} \label{eq:char} \big(g(v)\big)_{v\in Q_0}\mapsto\prod_{v\in Q_0} \big(\det g(v)\big)^{\delta(v)}.
\end{equation*}
A {\em semi-invariant} function of weight $\delta$ is an element in
$$\SI_\alpha(A)_\delta:=\{s\in k[\rep_\alpha(A)]\mid g(s)=\chi_{\delta}(g)s,\ \forall g\in\GL_\alpha \}.$$
The graded semi-invariant algebra associated to $\delta$ is 
$$\SI_\alpha^\delta(A):= \bigoplus_{n\geq 0} \SI_\alpha(A)_{n\delta}.$$
A representation $M$ is called {\em $\delta$-semistable} if there is some $s\in \SI_\alpha^\delta(A)$ such that $s(M)=0$.

\begin{lemma}[{\cite[Proposition 3.1]{Ki}}] \label{L:King} A representation $M$ is {\em $\delta$-semistable} (resp. {\em $\delta$-stable}) if and only if $\delta(\dv M)=0$ and $\delta(\dv L)\leq 0$ (resp. $\delta(\dv L)<0$) for any non-trivial subrepresentation $L$ of $M$.
\end{lemma}
%\noindent It turns out () that $\delta$-semistable is equivalent to $\chi_\delta$-semistable.
%He derived a nice criterion for the (semi)-stability of a representation.
%Here we state his criterion as our definition for the stability.
%Since $Q$ has no oriented cycles, the degree zero component is the field $k$ \cite{Ki}.

For any projective presentation $d$ of weight $\delta$ such that $\delta(\alpha)=0$, Schofield constructed the following semi-invariant function of weight $\delta$ on $\rep_\alpha (A)$. 
We apply the functor $\Hom(-,M)$ to $d$ for $M\in\rep_\alpha(A)$
\begin{equation*} \label{eq:canseq} \Hom(P_+,M)\xrightarrow{C(d,M)}\Hom_Q(P_-,M).
\end{equation*}
Since $\delta(\alpha)=0$, $C(d,M)$ is a square matrix. We define
$$c_d(M):=\det C(d,M).$$

\begin{theorem}[{\cite[Theorem 1]{DW1},\cite{DWrel},\cite{SV,DZ}}] \label{T:Cdgen} The space $\SI_\alpha(A)_\delta$ is spanned by semi-invariants of the form  $c_d$ where $d$ has weight $\delta$.
\end{theorem}

\begin{lemma} \label{L:ss} A representation $M\in \rep_\alpha(A)$ is $\delta$-semistable if and only if 
	$$\hom(n\delta,M) = \delta(\alpha) =0\ \text{ for some $n\in\mb{N}$}.$$
	Moreover, $n$ can be replaced by $kn$ for any $k\in\mb{N}$.
	If $m$ is the minimum of all such $n$, then $m\delta$ can not be decomposed as $m\delta = k\delta \oplus (m-k)\delta$ for any $k$.
	In particular, if $\delta$ is not wild, then $m=1$.
\end{lemma}
\begin{proof} By Theorem \ref{T:Cdgen} $M$ is $\delta$-semistable if and only if $c_d(M)\neq 0$ for some $d$ of weight $n\delta$ for some $n\in\mb{N}$.
	This happens if and only if the matrix $C(d,M)$ is invertible, which is equivalent to that $\hom(d,M)=\e(d,M)=0$.
	The condition that $\hom(d,M)=\e(d,M)=0$ for some $d$ is clearly equivalent to that $\hom(n\delta,M) = \delta(\alpha) =0$ for some $n\in\mb{N}$.
	
	The moreover part follows easily from Lemma \ref{L:homineq}.
	If $\delta$ is not wild, then $n\delta = \delta\oplus \cdots \oplus \delta$ by Theorem \ref{T:CDPHom}.
\end{proof}

We remark that the first statement of \cite[Theorem 1.1]{BST} follows from a special case of this lemma when the $\delta$ is real or equivalently the general presentation of weight $\delta$ is rigid.

\subsection{Proof of Theorem \ref{T:HomE}} \label{ss:proof}
%\subsection{One-point Extension}
We need to review the notion of one-point extension of $A$.
Let $M$ be a right $A$-module. Treating $M$ as a $k$-$A$-bimodule, the triangular algebra $A[M]:=\sm{A & 0\\ M & k}$ is called (trivial) {\em one-point extension} of $A$ by $M$. There is an obvious dual notion of one-point coextension $A[M^*]:=\sm{k & 0\\ M^* & A}$.

Suppose that $M\in\rep A$ is presented by $P(\beta_-)\xrightarrow{d} P(\beta_+)\to M\to 0$.
Then the algebra $\wtd{A}=A[M]$ can be presented by a new quiver $Q(M)$, which is obtained from $Q$ by adjoining a new vertex $\boxminus$ and $\beta_+(v)$ new arrows from $\boxminus$ to the vertex $v\in Q$. The relations are clearly given by the presentation $d$. In reality, the presentation is always chosen to be minimal. 
%By abuse of notation, we also use $Q[M]$ to denote the new quiver $Q(M)$ with those new relations. 
The one-point coextension $A[M^*]$ can be similarly described using injective presentation of $M$. By convention, the newly adjoined vertex is denoted by $\boxplus$. 
\footnote{The notation suggest that $\boxminus$ and $\boxplus$ should be visualized as the frozen vertices with label $-$ and $+$.}
By construction, we have the following exact sequences
\begin{align} &0\to (M,0)\to P_\bminus\to S_\bminus \to 0 && \text{for one-point extensions,} \\
&0\to S_\bplus \to I_\bplus \to (0,M)\to 0 && \text{for one-point coextensions.}
\end{align}

Let $B$ either be the algebra $A[M]$ or the algebra $A[M^*]$.
We have a restriction functor $\res_A: \rep B \to \rep A$ sending $M$ to $Me$ where $e=1-e_{\pm}$ and $e_{\pm}$ is the idempotent corresponding to the vertex $\boxplus$ or $\boxminus$.
The restriction functor has two induction functors $T_B:=-\otimes_A eB$ and $L_B:=\Hom_A(Be,-)$.
\begin{lemma}[{\cite[Theorem I.6.8]{ASS}}] \label{L:adjoint} $T_B$ (resp. $L_B$) is left (resp. right) adjoint to $\res_A$.
	Moreover, they satisfy $\res_A T_B \cong \Id_{\rep A} \cong \res_A L_B$.
\end{lemma}

%Using the induction functor $V \mapsto \wtd{V}:=V\otimes_A A[M^*]$, we get a representation of $A[M^*]$ for each representation of $A$.
%There is also a dual induction functor $V \mapsto \wtd{V}^\vee:=\Hom_A(A[M],-)$.

\begin{corollary} \label{L:Tgen} If $d$ is a general presentation of weight $\delta$, then $T_{B}(d)$ is a general presentation of weight $(\delta,0)$ or $(0,\delta)$ (depending on $B=A[M]$ or $A[M^*]$).
	Moreover, $\Coker(T_B(d)) = T_B(\Coker(d))$.
\end{corollary}
\begin{proof} The first statement follows from the equality 
	$$\Hom_B(T_B(V),T_B(W)) \cong \Hom_A(V,\res_AT_B(W)) \cong \Hom_A(V,W) \text{ for any $V,W\in\rep A$}.$$
	The second statement is due to the right exactness of $T_B$.
\end{proof}

Let $M$ be a representation of $A$. We extend $A$ by $M$ and obtain the algebra $A^-:=A[M]$.
Then we coextend $A[M]$ by the indecomposable projective representation $P_\bminus$ of $A[M]$, and obtain the algebra $(A[M])[P_\bminus^*]$.
We denote $A^\pm:=(A[M])[P_\bminus^*]=A^-[P_\bminus^*]$.
Note that 
$$A^\pm = \left(\begin{matrix} k & 0 & 0\\ M^* & A & 0\\ k & M & k
\end{matrix} \right)$$
Throughout we use $P_\bminus$ to denote the above indecomposable projective representation of $A[M]$ rather than the indecomposable projective representation of $A^\pm$.
%We denote by $(0,P_\bminus)$ the extension by zero of $P_\bminus\in \rep A^-$ to $\rep A^\pm$. % $(0,P_\bminus)$ is nothing but $L_{A^\pm}(P_\bminus)$.

%Since $\res_A(I_+)=M$, we have in particular that
%\begin{lemma} \label{L:Hom_res} $\Hom_{A[M^*]}(\wtd{V}, I_+) = \Hom_A(V, M)$ and $\Hom_{A[M]}(P_-,\wtd{V}^\vee) = \Hom_A(M,V)$.
%\end{lemma}	
%This functor can be described using presentations.
%Suppose that $V$ is minimally presented by $0\to V\to I_0\to I_1$, then $L_B(V)$ is presented by
% $0\to L_B(V)\to L_B(I_0)\to L_B(I_1)$.

\begin{lemma} \label{L:Hom_res} We have that 
	$$\Hom_{A^\pm}\big(T_{A^\pm}(T_{A^-}(V)), I_\bplus \big) \cong \Hom_A(V, M).$$
\end{lemma}	
\begin{proof} We have that $\res_{A^-}(I_\bplus) = P_\bminus$ and $\res_A(P_\bminus) = M$.
	So apply Lemma \ref{L:adjoint} twice, and we get
	$$\Hom_{A^\pm}\big(T_{A^\pm}(T_{A^-}(V)), I_\bplus \big) \cong \Hom_{A^-}(T_{A^-}(V), P_\bminus)\cong \Hom_{A}(V, M).$$
\end{proof}

\begin{definition} A vertex $v$ is called {\em maximal} in a representation $M$
	if $\dim M(v)=1$ and all strict subrepresentations of $M$ are not supported on $v$.
\end{definition}

Let ${f}_\bplus$ (resp. ${f}_{\bminus}$) be the tropical $F$-polynomial of $I_\bplus$ (resp. $P_\bminus$).
%Let $y_i$ be the linear functional $y_i(\delta)=\delta(i)$.
\begin{lemma} \label{L:max} $\boxminus$ is a maximal vertex of $I_\bplus \in \rep A^\pm$.
	Moreover, we have that 
	\begin{align*}
	{f}_{\bminus}((\delta,\delta_-)) &= \max({f}_M(\delta), \delta(\dv M)+\delta_- ),\\
	{f}_{\bplus}((\delta_+,\delta,\delta_-)) &= \max(0,{f}_{\bminus}((\delta,\delta_-)) + \delta_+).
	\end{align*}
\end{lemma}
\begin{proof} Recall that we have two exact sequences
	\begin{align*}0\to (M,0)  \to P_\bminus \to S_\bminus \to 0, \\
	0\to S_\bplus  \to I_\bplus \to (0,P_\bminus) \to 0. \end{align*}
	Since the $1$-dimensional subspace of $P_\bminus$ at vertex $\boxminus$ generates $P_\bminus$, 
	we see that a subrepresentation of $P_\bminus$ is either a subrepresentation of $M$ or $P_\bminus$ itself.
	Hence, $${f}_{\bminus}((\delta,\delta_-)) = \max({f}_M(\delta), \delta(\dv M)+\delta_- ).$$
	Next, whenever there is a subrepresentation $S$ of $P_\bminus$, we have a subrepresentation $(k,S)$ of $I_\bplus$.
    Conversely, any nonzero subrepresentation of $I_\bplus$ must be supported on $\boxplus$.
    Hence, $${f}_{\bplus}((\delta_+,\delta,\delta_-)) = \max(0,{f}_{\bminus}((\delta,\delta_-)) + \delta_+).$$
\end{proof}

%Here is the key lemma proved in \cite{Fs1} for $A$ being Jacobian algebras. The argument actually works for any finite-dimensional algebras. 
\begin{lemma}\label{L:hom=0} Suppose that a representation $M$ contains a maximal vertex $v$. Then $f_M(\delta)=0$
if and only if $\hom(n\delta,M)=0$ for some $n\in \mb{N}$.
\end{lemma}

\begin{proof} If $\hom(n\delta,M)=0$, then $\hom(n\delta,L)=0$, and thus $\delta(\dv L)\leq 0$ for all subrepresentations $L$ of $M$.
Conversely, suppose that $\PHom(\delta)=\Hom(P_-,P_+)$.
	We add $c=-\delta(\dv M)\geq 0$ copies of $P_v$'s to $P_+$ so that a presentation in $\Hom(P_-,P_+ \oplus cP_v)$ has weight $\delta'=\delta+ c \e_v$.
	It satisfies that $\delta'(\dv M)=0$ and $\delta'(\dv L)=\delta(\dv L)\leq 0$ for all subrepresentations $L \subsetneq M$.
	By King's criterion (Lemma \ref{L:King}), we see that $M$ is $\delta'$-semistable, and thus $\hom(n\delta',M)=0$ for some $n\in\mb{N}$ by Lemma \ref{L:ss}.
	Now a general presentation $nP_-\xrightarrow{(d,d')} nP_+\oplus ncP_v$ must have $d$ general in $\Hom(nP_-,nP_+)$.
	Hence, $\hom(n\delta',M)=0$ implies $\hom(n\delta,M)=0$.
\end{proof}

\begin{remark} This lemma was proved in \cite[Lemma 6.6]{Fs1} for $A$ being a Jacobian algebra. The argument actually works for any finite-dimensional algebras. Unfortunately, there is a gap in the proof where we assume that $n$ can always be $1$.
	However, each representation $T_{i,j}$ in \cite[Section 6]{Fs1} is negative reachable.
	So this would not affect the main results there due to Theorem \ref{T:HomEQP} below.
\end{remark}
%Let $I$ be a subset of $Q_0$. We denote by $Q|_I$ the full subquiver of $Q$ with vertex-set $I$.
%\begin{definition} 
%	We say a $\delta$-vector supported on $I$ or on $Q|_I$ if the generic cokernel $\Coker(\delta)$ is supported on $I$.
%\end{definition}

\begin{proof}[Proof of Theorem \ref{T:HomE}] (1). Let $d$ be general presentation of any weight $\delta$ in $\rep A$ and $V=\Coker(d)$. 	By Corollary \ref{L:Tgen}, $T_{A^\pm}(T_{A^-}(V)) = \Coker(T_{A^\pm}(T_{A^-}(d)))$ and $T_{A^\pm}(T_{A^-}(d))$ is a general presentation of weight $\tilde{\delta}=(0,\delta,0)$.
	By Lemma \ref{L:Hom_res}, we have that 
	$$\hom(\tilde{\delta},I_\bplus)=\hom(\delta,M).$$
	
	$\Coker(\tilde{\delta})$ may not be supported on the original quiver $Q$. We are going to put an appropriate negative weight $\delta_+$ on the vertex $\bplus$.
	By Lemma \ref{L:max}, we have that 
	$${f}_\bplus((\delta_+,\delta,0))= \max(0,{f}_{\bminus}((\delta,0)) + \delta_+) = \max(0, {f}_M(\delta) +\delta_{+}).$$
	Let $\delta_+ = -f_M(\delta)$, then ${f}_\bplus((\delta_+,\delta,0))=0$.
	By Lemma \ref{L:hom=0}, $\hom(n(\delta_+,\delta,0),I_\bplus) = 0$ for some $n\in \mb{N}$.
	But $\hom((n\delta_+,n\delta,0),I_\bplus) = 0$ implies that $\hom((0,n\delta,0),I_\bplus) \leq -n\delta_+$ because $\dim I_{\bplus}(\bplus) = 1$.
	Hence we have
	$$f_M(n\delta) \geq \hom((0,n\delta,0),I_\bplus)=\hom(n\delta,M).$$
	We get the equality by Lemma \ref{L:HomE}.
	Then the equality $\fc_M(-n\delta) = {\e}(n\delta,M)$ follows from the relation \eqref{eq:heform}.
	
	The other half can be proved by the dual argument.	
	The moreover part follows from the corresponding part in Lemma \ref{L:ss}.
\end{proof}

\subsection{The Case of Quivers with Potentials}
We refer the readers to the original papers \cite{DWZ1,DWZ2} for the theory of the quivers with potential. 
In this subsection, $(Q,\mc{P})$ is a quiver with potential such that its {\em Jacobian algebra} $A=J(Q,\mc{P})$ is finite-dimensional. 
The key notion introduced in \cite{DWZ1} is the {\em mutation} of a quiver of potential and its {decorated} representations.

A vertex is called {\em admissible} if it is not involved in any oriented cycle of length $\leq 2$.
For each admissible vertex $u \in Q_0$, there is an operation $\mu_u$, which yields a new quiver with potential $\mu_u(Q,\mc{P})$.
A {\em decorated} representation $\mc{M}=(M,V)$ consists of two parts: $M$ is a usual representation and the decorated part $V$ is a $k^{Q_0}$-module. A usual representation $M$ can be regarded as a decorated representation $(M,0)$.
For each (decorated) representation $\mc{M}$, there is a mutated representation $\mu_u(\mc{M})$ of $\mu_u(Q,\mc{P})$.
For any weight vector $\delta\in\mb{Z}^{Q_0}$, there is a mutated weight vector $\mu_u(\delta)$ defined by \cite[(2.11)]{DWZ2}.
For a decorated representation $\mc{M}=(M,V)$, its tropical $F$-polynomial and related functors, such as $\Hom(-,M)$ and $\E(-,M)$, are all defined to be those of $M$.  

\begin{lemma} \label{L:HomEQP} Let $\bs{\mu}$ be a sequence of mutations (at admissible vertices). We denote $M':=\bs{\mu}(M)$ and $\delta' := \bs{\mu}(\delta)$.
	We have the following relation for any representation $M$ and $\delta\in \mb{Z}^{Q_0}$:
	\begin{align*}
	&{f}_{M'}(\delta') - {f}_M(\delta) = \hom(\delta',M') - \hom(\delta,M);\\
	&\fc_{M'}(-\delta') - \fc_M(-\delta) = {\e}(\delta',M') - {\e}(\delta,M).
	\end{align*}
There are similar relations for $\fc_M(\dtc)$, $\hom(M,\dtc)$ and $\ec(M,\dtc)$.
\end{lemma}
\begin{proof} By induction it suffices to show for any one-step mutation $\mu_u$.
		We knew from \cite[Proposition 6.1]{DWZ2} that 
	\begin{equation} \label{eq:homdiff} \hom(\delta',M') - \hom(\delta,M) = [\delta'(u)]_+[\dtc_{M'}(u)]_+ - [\delta(u)]_+[\dtc_M(u)]_+.
	\end{equation}
	Recall from \cite[Lemma 5.2]{DWZ2} that 
	$$(1+y_u)^{h_u}F_M(\b{y}) = (1+y_u')^{h_u'}F_{M'}(\b{y}'),$$ 
	where $h_u = -[\dtc_M(u)]_+$ and $h_u' = -[\dtc_{M'}(u)]_+$.
Taking the Newton polytope (see the remark after Definition \ref{D:tropf}) \footnote{Equivalently we can take tropicalization here but we need the nontrivial positivity results.}, we get
	\begin{align*} &-[\delta(u)]_+[\dtc_M(u)]_+ + f_M(\delta) = -[\delta'(u)]_+[\dtc_{M'}(u)]_+ + f_{M'}(\delta'), \\
	\Rightarrow \quad & \quad f_{M'}(\delta') - f_M(\delta) = [\delta'(u)]_+[\dtc_{M'}(u)]_+ - [\delta(u)]_+[\dtc_{M}(u)]_+.
	\end{align*}
Compare with \eqref{eq:homdiff}, and we obtain the first relation.
The other relation follows easily from \eqref{eq:heform}.
\end{proof}

We say a representation $M$ of $(Q,\mc{P})$ {\em negative reachable} if there is a sequence of mutations $\bs{\mu}$ such that $\bs{\mu}(M)$ is {\em negative}, i.e., $\bs{\mu}(M)$ has only the decorated part.
\begin{theorem} \label{T:HomEQP} If $M$ is negative reachable, then for any $\delta,\dtc \in\mb{Z}^{Q_0}$ we have that
	\begin{align}
\label{eq:HomEQP}	{f}_M(\delta) &= \hom(\delta,M), & \fc_M(-\delta) &= {\e}(\delta,M);\\
\label{eq:HomEQPdual}	\fc_M(\check{\delta}) &= \hom(M,\check{\delta}), & {f}_M(-\check{\delta}) &= \ec(M,\check{\delta}).
	\end{align}
\end{theorem}
\begin{proof} By Lemma \ref{L:HomEQP}, it is enough to notice that if $M$ is negative, then $f_{M}(\delta) = \hom(\delta,M) = 0$ and $\fc_{M}(\dtc) = \hom(M,\dtc) = 0$ for any $\delta$ and $\dtc$.
\end{proof}

\begin{corollary} \label{C:dim} 
	If $I_i$ is negative reachable, then the dimension vector $\alpha$ of $\Coker(\delta)$ can be computed by
	$$\alpha(i)=f_{I_i}(\delta).$$
	If $P_i$ is negative reachable, then the dimension vector $\check{\alpha}$ of $\Ker(\dtc)$ can be computed by
	$$\check{\alpha}(i)=\fc_{P_i}(\check{\delta}).$$
\end{corollary}

\begin{example}\label{ex:dimfail} If $A$ is not a Jacobian algebra, then Corollary \ref{C:dim} may fail.
	We plug $M$ and $A=kQ$ in Example \ref{ex:n>1} into the construction of Section \ref{ss:proof}.
	Then from the proof of Theorem \ref{T:HomE} we see that $f_{I_\bplus}(\tilde{\delta}) = f_M(\delta)=0$ but $\hom(\tilde{\delta}, I_{\bplus}) = \hom(\delta,M) =1$.
	We also note that $\fc_{\Coker{\tilde{\delta}}}(e_{\bplus}) = \hom(\tilde{\delta}, I_{\bplus})=1$.
\end{example}

\begin{question} \label{q:HomEQP} Does the conclusion of Theorem \ref{T:HomEQP} still hold if $M=\Ker(\dtc)$ in \eqref{eq:HomEQP} and $M=\Coker(\delta)$ in \eqref{eq:HomEQPdual}?
\end{question}
\noindent This is certainly true for acyclic quivers due to Schofield's result (Theorem \ref{T:introext}).
By Lemma \ref{L:HomEQP} this is also true for mutation-acyclic QPs. 
Moreover, in this case $M=\Ker(\dtc)$ for some $\dtc$ iff $M=\Coker(\delta)$ for some $\delta$.
\begin{corollary} \label{C:HomEQP} If $(Q,\mc{P})$ is mutation-equivalent to an acyclic quiver, then 
	the conclusion of Theorem \ref{T:HomEQP} holds for $M=\Ker(\dtc)$ or $M=\Coker(\delta)$.
\end{corollary}

\section{Functors Associated to \texorpdfstring{$\delta$}{delta}} \label{S:tf}
In this section, we briefly review the two pairs of functors considered in \cite{Fc}.
\begin{lemma}[{\cite[Lemma 3.3]{Fc}}]\label{L:sub=h} Let $L$ be any subrepresentation of $M$. Then $\delta(\dv L)= \hom(\delta,M)$ if and only if $\hom(\delta,M/L)=\e(\delta,L)=0$.
	Moreover, if $L'$ is another such subrepresentation, that is,  $\delta(\dv L')=\hom(\delta,M)$, then both 
	$L\cap L'$ and $L+ L'$ are such subrepresentations.
\end{lemma}

Let $\mc{L}(\delta,M)$ be the set of all subrepresentations $L$ of $M$ such that $\delta(\dv L)=f_M(\delta)$.
\begin{theorem}[{\cite[Thereom 3.4]{Fc}}] \label{T:maxminsub} The set $\mc{L}(\delta,M)$ contains a unique minimal element $L_{\min}$ and a unique maximal element $L_{\max}$.
	Moreover, $L_1/L_0$ is $\delta$-semistable for any $L_0\subset L_1$ in $\mc{L}(\delta,M)$.
\end{theorem}

\begin{definition}[\cite{Fc}] \label{D:torsion} Let $(t_{\dtb},f_{\dtb})$ and $(\tc_{\dtb},\fc_{\dtb})$ be the pairs of functors associated to the torsion pair $({\mc{T}}(\dtb),{\mc{F}}(\dtb))$ and $(\check{\mc{T}}(\dtb),\check{\mc{F}}(\dtb))$, where
\begin{align*} {\mc{F}}(\dtb) = \{N\in \rep(A)&\mid \hom(n\delta,N) =0 \text{ for some } n\in\mb{N} \},\\
\check{\mc{T}}(\dtb) = \{L\in \rep(A) &\mid \e(n\delta, L) =0   \text{ for some } n\in\mb{N}  \}.
\end{align*}	
\end{definition}
\noindent If $\delta$ is not wild, by Lemma \ref{L:homineq} we can let $n=1$ in the definition of ${\mc{F}}(\dtb)$ and $\check{\mc{T}}(\dtb)$. In this case, the functors will be denoted by $(t_{{\delta}},f_{{\delta}})$ and $(\tc_{{\delta}},\fc_{{\delta}})$.

\begin{theorem}[{\cite[Theorem 3.10]{Fc}}] \label{T:torsion} We have that for any representation $M$ and any $\delta \in \mb{Z}^{Q_0}$,
	\begin{align*}t_{\dtb}(M)=L_{\min}\ &\text{ and }\  f_{\dtb}(M)=M/L_{\min};\\
	\tc_{\dtb}(M)=L_{\max}\ &\text{ and }\ \fc_{\dtb}(M)=M/L_{\max}.\end{align*}
	In particular, we have for any $L\in\mc{L}(\delta,M)$ that
	$$\Hom(t_\dtb(M),M/L) = 0\ \text{ and }\ \Hom(L,\fc_\dtb(M)) = 0.$$
\end{theorem}

Suppose that $\hom(M,N)=h$. We choose a basis of $\Hom(M,N)$ and take $hM \to N$ to be the canonical map with respect to this basis.
We call this map a {\em universal homomorphism} from $\op{add}(M)$ to $N$.

\begin{corollary}[{\cite[Corollary 3.13]{Fc}}] \label{C:univhomo} Suppose that $d$ is a rigid presentation with weight $\delta$.
	Then $t_\delta(M)$ is the image of the universal homomorphism $h\Coker(d) \to M$ while $\tc_\delta(M)$ is the kernel of the universal homomorphism $M \to e \Ker(\nu d)$,
	where $h=\hom(\delta,M)$ and $e=\e(\delta,M)$.
\end{corollary}

\section{Newton Polytopes of Representations} \label{S:Newton}

\subsection{A Presentation of \texorpdfstring{$\N(M)$}{N(M)}}
In this subsection we mostly follow \cite[Section 4.2]{BK}.
Let $V$ be a $\mathbb R$-vector space.
To a non-empty compact convex subset ${\sf P}$ of $V$, we associate its support function $\psi_{\sf P}:V^*\to\mathbb R$, 
which maps a linear function $f\in V^*$ to the maximal value $f$ takes on ${\sf P}$.
Then $\psi_{\sf P}$ is a sublinear function on $V^*$. 
One can recover ${\sf P}$ from the datum of $\psi_{\sf P}$ by the Hahn-Banach theorem \cite{HL}
$${\sf P}=\{v\in V\mid \alpha(v)\leq \psi_{\sf P}(\alpha),\ \forall\alpha\in V^*\},$$
and the map ${\sf P}\mapsto\psi_{\sf P}$ is a bijection from the set of all
non-empty compact convex subsets of $V$ onto the set of all sublinear functions on $V^*$. % (see for instance \cite[Chapter~C]{HUL}). 

\begin{theorem}\label{T:Newton} The Newton polytope $\N(M)$ is defined by 
	$$\{\gamma\in \mb{R}^{Q_0}\mid \delta(\gamma)\leq \hom(\delta,M),\ \forall \delta \in \mb{Z}^{Q_0} \}.$$
	The dual Newton polytope $\check\N(M)$ is defined by 
	$$\{\gamma\in \mb{R}^{Q_0}\mid \check{\delta}(\gamma)\leq \hom(M,\check{\delta}),\ \forall \check{\delta} \in \mb{Z}^{Q_0} \}.$$
\end{theorem}
\begin{proof} In our setting of ${\sf P}=\N(M)$, the support function is given by $\delta \mapsto f_M(\delta)$.
So $\N(M)$ is defined by 
$$\{\gamma\in \mb{R}^{Q_0}\mid \delta(\gamma)\leq f_M(\delta),\ \forall \delta \in \mb{R}^{Q_0} \}.$$
We know in priori that $\N(M)$ have integral vertices so its normal vector can be choosen as integral as well. It is enough to consider all $\delta\in \mb{Z}^{Q_0}$.
In general, we have that $f_M(\delta) \leq \hom(\delta,M)$.	
But $f_M(n\delta) = \hom(n\delta,M)$ for some $n\in \mb{N}$ by Theorem \ref{T:HomE}.
So $\delta(\gamma)\leq f_M(\delta)$ is equivalent to that $n\delta(\gamma)\leq f_M(n\delta) = \hom(n\delta,M)$. 
The presentation for $\N(M)$ follows.
\end{proof}

\noindent We know a priori that the Newton polytope has a (finite) hyperplane representation. 
In fact we only need those $\delta$-vectors which are outer normal vectors of $\N(M)$.
It is an interesting problem to find a finite set of $\delta$-vectors determining the Newton polytope.
This is achieved for general representations of any acyclic quiver in \cite{Fc}.

%\begin{proposition} For acyclic quivers, we have that $\N(\delta) = \N(n\delta)$ for any $n\in \mb{N}$.
%\end{proposition}

\subsection{Facets and Normals}
Recall that a $\delta$-vector is called {indecomposable} if a general presentation in $\PHom(\delta)$ is indecomposable. 
\begin{lemma} \label{L:linearity} Suppose that $\{\delta_1,\cdots, \delta_r\}$ satisfies $\e(\delta_i,\delta_j)=0$ for $i\neq j$. Then 
	$$f_M(\sum_i c_i\delta_i) = \sum_i c_if_M(\delta_i).$$
\end{lemma}
\begin{proof}     By Lemma \ref{L:homineq} we have that $\e(a\delta_i, b\delta_j)=0$ for $i\neq j$ and any $a,b\in\mb{N}$.
	We set $\delta := \sum_i c_i\delta_i$, then $\delta$ decomposes as $\delta = \bigoplus_i c_i\delta_i$ ($c_i\delta_i$ may be decomposable).
	By Theorem \ref{T:HomE}, there is some $\gamma\in\N(M)$ and $n\in\mb{N}$ such that 
    $n\delta(\gamma) = \hom(n\delta,M)$ for some $n\in \mb{N}$. 
    Note that $n\delta$ decomposes as $n\delta = \bigoplus_i nc_i\delta_i$.
	So
	$$\sum_i nc_i\delta_i(\gamma) =n\delta(\gamma)= \hom(n\delta,M)= \sum_i \hom(nc_i\delta_i,M),$$
	but each $nc_i\delta_i(\gamma) \leq \hom(nc_i\delta_i,M)$.
	Hence $nc_i\delta_i(\gamma)=\hom(nc_i\delta_i, M)=f_M(nc_i\delta_i)$ for each $i$, so $\delta_i(\gamma)=f_M(\delta_i)$.
	Then 
	$$\sum_i c_if_M(\delta_i) = \sum_i c_i\delta_i(\gamma) = \delta(\gamma)\leq f_M(\delta).$$
	Finally, the equality follows from the sublinearity of $f_M$.
\end{proof}

\begin{corollary} \label{C:indnormal} Let $\delta$ be an indivisible outer normal vector of $\N(M)$. Then in any decomposition $n\delta = \delta_1 \oplus \delta_2$, $\delta_i$ must be a multiple of $\delta$.
In particular, $\delta$ is indecomposable.
\end{corollary}
\begin{proof} 	Suppose that none of $\delta_1$ and $\delta_2$ is a multiple of $\delta$.	
	For any $\gamma \in \N(M)$ on this facet, we have that
	$$\delta_1(\gamma)+\delta_2(\gamma) = n\delta(\gamma)= f_M(n\delta)= f_M(\delta_1)+f_M(\delta_2).$$
	Since $\delta_i(\gamma)\leq f_M(\delta_i)$, we must have that $\delta_i(\gamma)= f_M(\delta_i)$ for $i=1,2$.
	This implies that both $\delta_1$ and $\delta_2$ are out normal vectors of this facet.
	A contradiction.
\end{proof}

\begin{definition} For a fixed algebra $A$, a weight vector $\delta$ is called {\em normal} if it is an outer normal vector of the Newton polytope of some $M\in \rep A$. % such that $f_M(\delta) = \hom(\delta,M)$.
\end{definition}

\begin{question} Is any indecomposable $\delta$-vector normal?
\end{question}
\noindent Later we shall see that each real indecomposable $\delta$-vector is normal. Moreover, if $A$ has no relations, then each indecomposable $\delta$-vector is normal.
 
%The importance of this question is reflected on the following proposition, which partially answered Question \ref{q:wild}.
%It is also a reasonable generalization of Schofield's result (\cite[Theorem 3.7]{S2}), which says that if $\beta$ is imaginary Schur root then so is $m\beta$ for any $m\in \mb{N}$.
%\begin{proposition}\label{P:wild} If $\delta$ is a wild normal vector, then $m\delta$ is wild for each $m\in\mb{N}$.
%\end{proposition}
%\begin{proof} Suppose that $\e(\delta,\delta)>0$. Then by Corollary \ref{C:me} $\e(m\delta,n\delta)>0$ for any $m,n\in\mb{N}$.
%	It remains to show that $m\delta$ is indecomposable. 	
%	By Theorem \ref{T:CDPHom} any multiple of $\delta$ can not appear in the canonical decomposition of $m\delta$.
%	Suppose that we have a decomposition $m\delta = \delta_1 \oplus \delta_2$, and $\delta_i$ is not a multiple of $\delta$.
%	Then the same proof as Corollary \ref{C:indnormal} shows this is impossible.
%\end{proof}

\begin{definition} Suppose that $\delta = \bigoplus_i \delta_i$ is the canonical decomposition of $\delta$, and the dimension of the subspace spanned by $\{\delta_i\}_i$ is $r$. We say that $\rep A$ has enough $\delta$-stable representations if there are $|Q_0|-r$ $\delta$-stable representations with linearly independent dimension vectors.	
\end{definition}
\noindent This is equivalent to say that the dimension vectors of $\delta$-semistable representations span a codimension $r$ subspace in $K_0(\rep A)$.

\begin{proposition}\label{P:normalstable} An indecomposable $\delta$ is normal if and only if $\rep A$ has enough $\delta$-stable representations.
\end{proposition}
\begin{proof} If $\delta$ is a normal vector of $\N(M)$, then the convex hull of dimension vectors in $\mc{L}(\delta,M)$ has codimension $1$.
By Theorem \ref{T:maxminsub} $L/L_{\min}$ is $\delta$-semistable for any $L\in \mc{L}(\delta,M)$,
and the dimension vectors of $L/L_{\min}$ span a codimension $1$ subspace.

Conversely, if $\rep A$ has $|Q_0|-1$ $\delta$-stable representations $\{L_i\}_i$ with linearly independent dimension vectors, then let $M=\bigoplus_i L_i$.
We claim that $\delta$ is a normal vector of $\N(M)$.
Since $M$ is $\delta$-semistable, we have that $\hom(n\delta,M)=0$ for some $n\in\mb{N}$. 
So $\{\gamma\mid \delta(\gamma)=0\}$ supports a face of $\N(M)$. 
Since each $L_i$ lies on this face, its codimension is exactly $1$.
\end{proof}

%\begin{remark}. Bridgeland 
%\end{remark}

One of the main results in \cite{Fc} gives an explicit formula for the restriction of the $F$-polynomial $F_M$ to a facet of its Newton polytope.
In particular, this result specializes to the tropical setting.
Roughly speaking, any facet of Newton polytope $\N(M)$ is a shifted Newton polytope $\N(M')$ for a representation $M'$ of another algebra. We refer the readers to \cite[Section 6]{Fc} for more details.

\subsection{Vertices and Dual Cones} \label{ss:Vcones}
If ${\sf P}$ is a polytope, then its support function is piecewise linear. 
The maximal regions of linearity of $\psi_{\sf P}$ are exactly the {\em dual cones} of the vertices of ${\sf P}$: 
for each vertex $v$ of ${\sf P}$, the support function $\psi_{\sf P}$ is linear on
$\{\alpha\in V^*\mid  \alpha(v) = \psi_{\sf P}(\alpha)\}$.
The extremal rays of the dual cone are precisely the normal vectors of all facets of ${\sf P}$ containing $v$.
For this reason it is also called the {\em normal cone} of $v$.
In our setting, the dual cone $\F_\gamma(M)$ of a vertex $\gamma\in \N(M)$ is the cone spanned by $\delta$ satisfying
\begin{equation} \label{eq:v2fan} \delta(\gamma) = f_M(\delta)
\end{equation}
Similarly, the dual cone $\Fc_\gamma(M)$ of a vertex $\gamma\in \Nc(M)$ is the cone spanned by $\dtc$ satisfying 
$$\dtc(\gamma) = \fc_M(\dtc).$$

Let ${\sf V}(M)$ and $\check{\sf V}(M)$ be the set of vertices in $\N(M)$ and $\check{\N}(M)$.
We first recall some results in \cite{Fc}.
\begin{proposition}[{\cite[Lemma 4.1, Proposition 4.7]{Fc}}] \label{P:vertex} $\gamma\in {\V}(M)$ if and only if it is the dimension vector of $t_\dtb(M)$ or $\check{t}_\dtb(M)$ for some weight $\delta \in \mb{Z}^{Q_0}$.
In particular, there is a unique subrepresentation $L$ of $M$ of dimension $\gamma$, and it satisfies $\Hom(L,M/L)=0$.	
\end{proposition}
\noindent It is quite clear that $\delta$ can be any weight in the interior of $\F_\gamma(M)$.
The converse of the last statement is not true. 
%However, for a {\em general} representation $M$ of an acyclic quiver, 
%if there is a unique subrepresentation of $M$ of dimension $\gamma$, then $\gamma \in {\V}(M)$ \cite{Fc}.

\begin{definition} For any $\gamma\in \V(M)$, we call the unique subrepresentation $L$ with $\dv L = \gamma$ a {\em vertex subrepresentation} of $M$.
	The {\em vertex quotient} representation is defined analogously.
\end{definition}
\noindent In particular, we can label the vertices of $\N(M)$ by the vertex subrepresentations of $M$.

\begin{corollary} \label{C:sumvertex} Suppose that $M=\bigoplus_i M_i$. Then 
each vertex subrepresentation $L$ of $M$ is of the form
$L=\bigoplus_i L_i$ where each $L_i$ is a vertex subrepresentation of $M_i$. 
In particular, $\N(M) = \sum_i \N(M_i)$ where the sum on the right side is the Minkowski sum.
\end{corollary}

Consider the sets 
\begin{align*}
{\Delta}_0(M)&=\{\delta\in \mb{Z}^{Q_0} \mid \hom(n\delta,M)=0 \text{ for some $n\in\mb{N}$} \}, \\
{\Delta}_1(M)&=\{\delta\in \mb{Z}^{Q_0} \mid {\e}(n\delta,M)=0 \text{ for some $n\in\mb{N}$} \}.
%\intertext{and dually}
%{\Delta}_0(M)=\{\delta \mid {f}_{M}(\delta)=0 \} \text{ and } {\Delta}_1(M)=\{\delta \mid \fc_{M}(-\delta)=0 \}.
\end{align*}
\noindent They span the two most important dual cones, namely $\F_0(M)$ and $\F_M(M)$.
We call them the major cones of $\N(M)$.
Clearly $\F_0(M)$ always contains the negative cluster $(-e_1,\dots,-e_n)$ and $\F_M(M)$ always contains the positive cluster $(e_1,\dots,e_n)$.
Moreover, $\Delta(M):=\Delta_0(M)\cap \Delta_1(M)$ consists of all weights $\delta$ such that $M$ is $\delta$-semistable.
Due to the relation $f_M(\delta)- \fc_M(-\delta)=\delta(\dv M)$, we have the obvious duality 
$$\F_L(M) = -\check{\F}_{M/L}(M).$$
It follows from Theorem \ref{T:HomE} that 
\begin{corollary} \label{C:HE0} ${\Delta}_0(M)$ $($resp. ${\Delta}_1(M)$$)$ are precisely the lattice points in the polyhedral cone defined by $\delta(v) \leq 0$ for all $v \in {\V}(M)$ $($resp. $\delta(v) \geq 0$ for all $v \in {\sf \check{V}}(M)$$)$.
\end{corollary}
\noindent %Readers can easily formulate the polyhedral conditions for the vanishing of $\hom(M,\check{\delta})$ and $\ec(M,\check{\delta})$.
One interesting result in \cite{Fc} says that if $M$ is a general representation of an acyclic quiver, 
then the normal vectors of $\N(M)$ are precisely given by the extremal rays in $\F_0(M)$ and $\F_M(M)$.

The following proposition says that other dual cones are intersections of the major cones.
\begin{proposition} \label{P:othercones} Suppose that $L$ is a vertex subrepresentation of $M$. We have that 
	$$\F_L(M) = \F_0(M/L) \cap \F_{L}(L).$$
\end{proposition}
\begin{proof} If $\delta\in \F_L(M)$, then $\delta(\dv L) = f_M(\delta)$.
	Since every subrepresentation of $L$ is a subrepresentation of $M$, $\delta(\dv L)\leq f_L(\delta)\leq f_M(\delta)$.
	So $\delta\in \F_L(L)$. Similarly we can show that $\delta\in \F_0(M/L)$.
	Conversely, if $\delta \in \F_0(M/L)\cap \F_L(L)$, then by Theorem \ref{T:HomE} there is some $n\in \mb{N}$ such that
	$\e(n\delta,L)=0$ and $\hom(n\delta,M/L)=0$. 
	By Lemma \ref{L:sub=h} we have that $n\delta(\dv L) = \hom(n\delta,M) = f_M(n\delta)$, i.e., $\delta\in \F_L(M)$.
\end{proof}

\begin{lemma} \label{L:refine} Suppose that $M=\bigoplus_i M_i$, and $\delta \in \F_L (M)$.
	Then $\delta \in \F_{L_i} (M_i)$ for each $i$ where $L_i = L\cap M_i$.
	So each $\F_{L_i}(M_i)$ is a union of dual cones of $\N(M)$. 
\end{lemma}
\begin{proof} By Corollary \ref{C:sumvertex} each $L_i$ is a vertex subrepresentation. 
	$\delta \in \F_L (M)$ implies that $\delta(\dv L)= f_M(\delta)$.
	So we have that
	$$\sum_i \delta(\dv L_i) = \delta(\dv L) = f_M(\delta) = \sum_i f_{M_i}(\delta).$$	
	But $\delta(\dv L_i)\leq f_{M_i}(\delta)$ for each $i$.
	We must have that $\delta(\dv L_i)= f_{M_i}(\delta)$ for each $i$.
	Hence $\delta \in \F_{L_i} (M_i)$.
	Conversely, suppose that $\delta \in \F_{L_i} (M_i)$ for each $i$. Then $\delta\in F_{\oplus_i L_i}(M)$.
	It follows that each $\F_{L_i}(M_i)$ is a union of $\F_L(M)$ where $L\cap M_i = L_i$.
\end{proof}

\begin{lemma} \label{L:clusterincone} Suppose that $\{\delta_1,\cdots, \delta_r\}$ satisfies $\e(\delta_i,\delta_j)=0$ for $i\neq j$.
Then all $\delta_i$'s are contained in some dual cone of ${\N}(M)$.
\end{lemma} %\marginpar{careful about the definition of generalized cluster}
\begin{proof} This is just a reformulation of Lemma \ref{L:linearity}.
\end{proof}

\begin{lemma} \label{L:erigid0} Let $M$ be an $\E$-rigid representation with weight vector $\delta$, and $N$ is a quotient representation of $M$.	Then $\delta(\dv N)=0$ if and only if $N=0$.
\end{lemma}
\begin{proof} Suppose that $N\neq 0$, then $\hom(M,N)>0$. 	Since $M$ is $\E$-rigid, $\e(M,N)\leq \e(M,M)=0$.
	We have that $\delta(\dv N) = \hom(M,N) - \e(M,N) > 0$. 
\end{proof}

%\begin{lemma} \label{L:seperation} Suppose that $\e(\delta_-,\delta_+)>0$ and $\e(\delta_+,\delta_-)=0$ with $M_+ = \Coker(\delta_+)$ and $M_-= \Ker(\nu \delta_-)$. Then $\delta_-$ and $\delta_+$ cannot lie in the same dual cone of $\F(M_- \oplus M_+)$.
%\end{lemma} \marginpar{why need $M_-$?}
%\begin{proof}  We have that $\hom(\delta_+, M_-)=\e(\delta_-,\delta_+)>0$.
%	Suppose that $\delta_+ \in \F_L(M_- \oplus M_+)$. By Lemma \ref{L:refine} $\delta_+ \in \F_{L_-}(M_-)\cap \F_{L_+}(M_+)$ where $L_\pm=L\cap M_\pm$.
%	Then by Lemma \ref{L:erigid0} $L_+$ has to be $M_+$.
%	But 
%	$\delta_-(\dv M_+) < \hom(\delta_-, M_+)$ because $\e(\delta_-, M_+)>0$.
%	Hence $\delta_- \notin \F_{L}(M_- \oplus M_+)$.
%\end{proof}

\begin{lemma} \label{L:seperation} Suppose that $\delta_-$ and $\delta_+$ are real, and $\e(\delta_-,\delta_+)>0$. Then $\delta_-$ and $\delta_+$ cannot lie in the same dual cone of $\F(M)$ where $M=\Coker(\delta_+)$.
\end{lemma} 
\begin{proof} Since $\delta_+$ is real, by Lemma \ref{L:erigid0} $\delta_+ \in \F_{M}(M)$ and $\delta_+ \notin \F_{L}(M)$ if $L\neq M$.
	But $\delta_-(\dv M) < \hom(\delta_-, M)$ because $\e(\delta_-, M)>0$.
	Since $\delta_-$ is real, $f_M(\delta_-)=\hom(\delta_-, M)$.
	Hence $\delta_- \notin \F_{M}(M)$.
\end{proof}

\begin{theorem} \label{T:Mcones} Let $\bs{\delta}_1, \dots, \bs{\delta}_m$ be finitely many clusters.
	Then there is some representation $M$ such that each $\bs{\delta}_i$ spans a dual cone of $\N(M)$.
\end{theorem}
\begin{proof} By Lemmas \ref{L:refine} and \ref{L:clusterincone} it suffices to show for a single cluster, say $\bs{\delta}=(\delta_1,\dots,\delta_n)$.
Let $(\delta_-, \delta_+)$ be an exchange pair. In particular, we have that 
$\e(\delta_-,\delta_+)>0$.
By Lemma \ref{L:seperation} there is a representation $N$ separating $(\delta_-,\delta_+)$ in the sense that they lie in two different dual cones of $\N(N)$.	
Let $N_i$ be the representation separating the (unordered) exchange pair $\{\delta_i,\delta_i'\}$ wrt. $\bs{\delta}\setminus \{\delta_i\}$.
By Lemma \ref{L:refine}	and \ref{L:clusterincone} $M=\bigoplus_i N_i$ is the desired representation. 
\end{proof}

\begin{remark} The proof shows that $M$ can be chosen to be a direct sum of $\E$-rigid representations. Later we will see in Corollary \ref{C:CS} that $M$ can be a direct sum of real Schur representations in the dual clusters under some mild assumption.
This theorem also implies that in particular each real indecomposable weight vector is normal. 	
\end{remark}

%\begin{lemma} Consider the exchange pair $(\delta, \tau^{-1}\delta)$.
%Then they lie in the same dual cone of $\F(M)$ for any $M\neq \Coker(\delta)$.	
%\end{lemma}
%\begin{proof} Suppose that some $\E$-rigid representation $E:=\hat{E}$ is missing.
%	%WLOG we assume that $(E_-,E_+)$ is an exchange pair with middle term $E_0$.
%	Then $(E_-=\tau^{-1} E, E)$ is an exchange pair.
%	We suspect $(E_-,E_+)$ lies in the same dual cone.
%	It suffices to show that $(E_-,E_+)$ lies in the same dual cone of $\F(M)$ for any $M\neq E_+$.
%	
%	Suppose that $(\delta_+,\delta_0) \in \F_L(M)$. We have 
%	$$\delta_+(L) = \hom(\delta_+,M),\quad \delta_0(L) = \hom(\delta_0,M)$$
%	We need to show that 
%	$$\delta_-(L) = \hom(\delta_-,M).   (*)$$
%	
%	%In particular, we have that $\delta_+(L) = \hom(\delta_+,E_+)=1$ so
%	%$L\cap E_+ = E_+$ by ??.
%	On the other hand, $\delta_+(L)+\delta_-(L)=\delta_0(L)$.
%	Moreover, we have the triangle
%	$d_+ \to d_0\to d_-\to d_+[1]$.
%	So  
%	$$0\to \hom(\delta_-,M)\to \hom(\delta_0,M) \to \hom(\delta_+,M) \to \e(\delta_-,M)\to \e(\delta_0,M)\to \e(\delta_+,M) \to 0.$$
%So (*) is equivalent to 
%%  $$
%\end{proof}

\section{Generic Newton Polytopes} \label{S:generic}

\subsection{Generic Newton Polytopes of \texorpdfstring{$\delta$}{delta}}
We first extend the notation $\hom(\delta,M)$ and $\hom(M,\dtc)$ in Definition \ref{D:home} in an obvious manner.
We write 
$$\hom(\delta,\dtc) := \hom(\delta,\Ker(\dtc)) = \hom(\Coker(\delta),\dtc).$$
Similarly, we write
$$f_\dtc(\delta) := f_{\Ker(\dtc)}(\delta)\ \text{ and }\ \fc_{\delta}(\dtc) = \fc_{\Coker(\delta)}(\dtc).
\footnote{It is not so easy to confuse this notation with the functors introduced in Section \ref{S:tf} under approprieate context.}$$
As we have seen in Example \ref{ex:dimfail} that $f_\dtc(\delta)\neq \fc_{\delta}(\dtc)$ in general even if one of $\delta$ and $\dtc$ is real.
We denote by $\N(\dtc)$ the Newton polytope of the kernel of a general presentation in $\IHom(\dtc)$.
We hope to determine $\N(\dtc)$ when $A$ is the Jacobian algebra of a quiver with potential.
The idea is based on the following observation.
In the rest of this section we assume $A$ is a Jacobian algebra of some QP.
\begin{observation} \label{O:vertex} According to Lemma \ref{L:clusterincone}, any cluster $\{\delta_i\}_i$ lies in some dual cone $\F_L(M)$ of $\N(M)$.
	Such a cluster determines the vertex $\dv L$ by the formula $\delta_i(\dv L) = \hom(\delta_i,M)$ for each $i$.
	The vertex can be explicitly computed as 
	$$\dv L = \underline{h} \bs{\delta}^{-1}, $$
	where $\bs{\delta}$ is the matrix $(\delta_1^{\T},\delta_2^{\T},\dots,\delta_n^{\T})$ and $\underline{h}(i)= \hom(\delta_i,M)$.
	In general, computing $\hom(\delta_i,M)$ is not easy. 
	However, when the cluster is negative reachable and $M$ is a generic kernel of $\IHom(\dtc)$, Theorem \ref{T:HomEQP} implies that
	$$\hom(\delta_i,M) = \hom(\delta_i,\dtc) = \fc_{\delta_i}(\dtc).$$
	Moreover, the tropical $F$-polynomial $\fc_{\delta_i}$ or equivalently $\Nc(\delta_i)$, the dual Newton polytope of $\Coker(\delta_i)$, may be computed by the mutation algorithm \cite{FZ4,DWZ2}.
\end{observation}

Now the question is whether each dual cone of $\N(\dtc)$ contains a cluster.
In fact, according to Theorem \ref{T:maxrigid} and Lemma \ref{L:clusterincone}, the question is equivalent to 
whether each dual cone of $\N(\dtc)$ contains a real $\delta$-vector.

\begin{question} \label{q:realincone} Does each dual cone of $\N(\dtc)$ contain a real $\delta$-vector?
\end{question}
\noindent We shall give a positive answer for acyclic quivers. This is based on the following lemma.

Recall that for any acyclic quiver $Q$ we can associate each dimension vector $\alpha$ a weight $\dtc_\alpha:=-\innerprod{-,\alpha}\in(\mb{Z}^{Q_0})^{*}$ where $\innerprod{-,-}$ is the Euler form of $Q$. 
We denote by $\N(\alpha)$ the Newton polytope of a general $\alpha$-dimensional representation.
Since there is an open set of $\rep_\alpha(Q)$ in which the representations have minimal injective presentations of weight $\dtc_\alpha$,
we have that $\N(\dtc_\alpha)=\N(\alpha)$.

\begin{lemma}[{\cite[Lemma 6.17]{Fc}}] \label{L:nonimg} Let $M$ be a general representation in $\rep_\alpha(Q)$.
	If $\delta$ corresponds to an imaginary root and $\hom(\delta,M)>0$, then
	the convex hull of the dimension vectors in $\mc{L}(\delta,M)$ has codimension at least $2$.
	In particular, such a $\delta$-vector cannot be a normal vector of $\N(M)$.
\end{lemma}

\begin{theorem}\label{T:genericN} Let $\alpha$ be any dimension vector of $Q$.
	Each normal cone $\F_\gamma(\alpha)$ of $\N(\alpha)$ contains a cluster.
	Hence the Newton polytope $\N(\alpha)$ is completely determined by the Newton polytopes of real Schur representations.
\end{theorem}

\begin{proof} If $\gamma=0$, $\F_\gamma(\alpha)$ contains the negative cluster.
	If $\gamma\neq 0$, it must be contained in some facet supported by $\{\gamma\in \mb{R}^{Q_0} \mid \delta(\gamma)= h>0\}$.
	Its normal vector $\delta$ can not be imaginary by Lemma \ref{L:nonimg}.
	So one ray of the cone is real.
	By the above remark it must contain a cluster.
\end{proof}

\begin{algorithm} For a fixed dimension vector $\alpha$ of $Q$, the (primitive) normal vectors of $\N(\alpha)$ are bounded.
It may be hard to give a sharp bound, but it is easy to estimate some rough bound.
Let $\Delta_\alpha$ be the set of all real $\delta$-vectors within this bound. 
We can use the mutation algorithm (\cite[Proposition 5.1]{FZ4}) to find the tropical $F$-polynomial of any $\delta\in \Delta_\alpha$.
In general, using the mutation algorithm to compute the $F$-polynomial of $\delta$ is very expensive, 
but it is much cheaper to find the tropical one by the tropical version of \cite[Proposition 5.1]{FZ4}.
Since the exchange graph of acyclic quivers are connected \cite{HU}, searching for all $\delta$ in $\Delta_\alpha$ can be terminated in finite steps.
Finally according to the above theorem, the generic Newton polytope $\N(\alpha)$ are determined by these tropical $F$-polynomials.
\end{algorithm}

\begin{example} \label{ex:genacyclic} Let $Q$ be the quiver $\twoone{1}{2}{3}$, and $\alpha$ be the dimension vector $(3,5,2)$.
	Except for zero and itself, the Newton polytope $\N(\alpha)$ has 4 vertices, which are listed in the left column.
	The middle column is one of the clusters determining the vertex, and the right column is the sequence of mutations to reach this cluster.
	\begin{align*} &(0,3,0) && (-e_1,e_2-e_3,-e_3) && 2\\
	&(0,0,2) && (-e_1,-e_2,e_3) && 3\\
	&(0,5,2) && (-e_1,e_2-e_3,e_2) && (2,3)\\
	&(2,3,2) && (3e_1-2e_2,2e_1-e_2,e_3) && (3,2,1,2,1)
%	(3,5,2) & 
	\end{align*}
\end{example}

Let us test Question \ref{q:realincone} in a very simple example.
\begin{example} \label{ex:QP4} Consider the quiver $\cyclicfourone{1}{4}{2}{3}{a}{b}{c}$ with potential $abc$.
	Let $\dtc=(0,-3,1,1)$. One can check that $\dtc$ is not real, and $M=\Ker(\dtc)$ has dimension vector $(1,1,1,2)$.
	Except for zero and itself, the Newton polytope $\N(\dtc)$ has 6 vertices.
		\begin{align*} &(0,0,0,1) && (-e_1,-e_2,-e_3,e_4-e_3) && 4\\ %[1,2,3,7]
		&(0,0,1,0) && (-e_1,-e_2,e_3-e_1,-e_4) && 3\\  % [1,2,5,4]
		&(0,0,1,2) && (-e_1,-e_2,e_4,e_4-e_3) && (4,3) \\ % [1,2,11,7]
		&(1,0,0,1) && (e_1,-e_2,-e_3,e_4-e_3) && (4,1)\\ %[14,2,3,7]
		&(1,0,1,1) && (e_1,-e_2,e_3,e_1-e_4) && (1,4,3)\\ %[14,2,12,8]
		&(1,0,1,2) && (e_1,-e_2,e_4,e_4-e_3) && (4,1,3) % [14,2,11,7]
		\end{align*}
%	\begin{align*} &(0,1,0,0) && [1,6,3,4] && 2\\
%	&(0,1,0,1) && [1,7,3,10] && (4,2)\\
%	&(0,1,1,1) && [1,6,9,16] && (2,?,?) \\
%	&(1,1,0,0) && [13,6,3,4] && (2,1)\\
%	&(1,1,0,2) && [3,7,10,19] &&\\
%	&(1,1,1,1) && [6,9,16,17] &&
%	\end{align*}
\end{example}

\begin{conjecture} \label{C:FGpairing} We have that $f_{\dtc}(\delta) = \fc_{\delta} (\dtc)$ for any $\delta$ and $\dtc$.
\end{conjecture}
\noindent A more optimistic conjecture is that $f_{\dtc}(\delta) = \fc_{\delta} (\dtc) = \hom(\delta,\dtc)$ (see Question \ref{q:HomEQP}).
%\begin{remark} Conjecture \ref{C:FGpairing} may fail if $A$ is not an Jacobian algebra.	
%\end{remark}

\begin{observation}\label{O:real2FG} The positive answer to Question \ref{q:realincone} implies Conjecture \ref{C:FGpairing}.
	If this is the case, we can determine each vertex of $\N(\dtc)$ using the method described in Observation \ref{O:vertex}.
\end{observation}
\begin{proof} Let $M=\Ker(\dtc)$. Suppose that $\delta \in \F_{\gamma}(M)$ and $\{\delta_i\}_i$ is a cluster in $\F_{\gamma}(M)$.
By Theorem \ref{T:maxrigid} we can write $\delta$ as an integral linear combination of $\delta_i$'s: $\delta = \sum_i c_i \delta_i$.
Then we have the following equalities, where the second one and the last one are due to Theorems \ref{T:HomE} and \ref{T:HomEQP} respectively.
$$\delta_i(\gamma) = f_{\dtc}(\delta_i) =\hom(\delta_i,\dtc)= \fc_{\delta_i} (\dtc).$$ 
Then we have the following equalities, where the fourth one is due to Lemma \ref{T:CDPHom} and Lemma \ref{L:directsum}.
$$f_{\dtc}(\delta) = \delta(\gamma) = \sum_{i} c_i\delta_i(\gamma) = \sum_i c_i \fc_{\delta_i} (\dtc) =  \fc_{\sum_i c_i \delta_i} (\dtc) = \fc_{\delta} (\dtc).$$
In fact, they both equal to $\hom(\delta,\dtc)$ by Lemma \ref{L:homineq}.
\end{proof}
\noindent Due to Theorem \ref{T:genericN} we have Schofield's Theorem \ref{T:introext} as a corollary of our Theorem \ref{T:HomE} and Observation \ref{O:real2FG}.

\begin{remark}[Relation to the cluster algebras]\label{R:sameNewton}	Determine $\N(\dtc)$ when $A$ is a Jacobian algebra is an important problem in the cluster algebra theory. Let $Q$ be a {\em 2-acyclic} quiver, and $B$ be its associated skew-symmetric matrix given by
	$$B(u,v) = |\text{arrows } u\to v| - |\text{arrows } v\to u|.$$ 
We denote by $\br{\mc{C}}(Q)$ the associated {\em upper cluster algebra} \cite{BFZ}. Let $(Q,\mc{P})$ be a nondegenerate quiver with potential. We still keep the assumption that $A=J(Q,\mc{P})$ is finite-dimensional.
	
In \cite{D} the author introduced a set of elements $\{X_\dtc\}$ indexed by the $\dtc$-vectors (or $\delta$-vectors), of the form 
$$X_\dtc = \b{x}^{-\dtc} F_{\dtc}(\b{y}),$$
where $F_{\dtc}$ is the $F$-polynomial of $\Ker(\dtc)$ \cite{DWZ2} and $\b{y}$ is a monomial change of variables from $\b{x}$: 
$y_u = \prod_{v} x_v^{B(u,v)}$.
In many cases they are turned out be a basis of $\br{\mc{C}}(Q)$ \cite{P}.	 
The Newton polytope of this polynomial $F_\dtc$ is exactly given by the generic Newton polytope $\N(\dtc)$.

In the meanwhile, a remarkable {\em positive} basis consisting of {\em theta functions} for all cluster algebras was introduced in \cite{GHKK}. 
For each $\dtc$-vector, there is a theta function $\vartheta_\dtc$, 
which is of the form $$\vartheta_\dtc = \b{x}^{-\dtc} \varphi_{\dtc}(\b{y}).$$
In general, the theta function can be a Laurent series, but let us assume it is a Laurent polynomial
so $\varphi_{\dtc}$ is a polynomial with {positive} coefficients.
Another very interesting positive (quantum) basis called {\em triangular basis} was introduced in \cite{Q} as a far-reaching generalization of \cite{BZ}. It has a similar form 
$$T_{\dtc,q} = \b{x}^{-\dtc} \psi_{\dtc,q}(\b{y}).$$
In particular, $\varphi_{\dtc}$ and $\psi_{\dtc,q}$ can be tropicalized and the tropicalization is determined by its Newton polytope.
%The above basis can also be indexed by the $\delta$-vectors.
We expect that all ``interesting" bases of cluster algebras should contain the cluster monomials and their tropicalizations satisfy the Conjecture \ref{C:FGpairing}. In light of Theorem \ref{T:genericN}, we have the following conjecture
\end{remark}
\begin{conjecture}\label{C:equalNewton} The Newton polytopes of $\varphi_\dtc$ and $\psi_{\dtc,q}$ are the same as the generic Newton polytope $\N(\dtc)$. Moreover, the coefficients in $F_{\dtc},\ \varphi_\dtc$, and $\psi_{\dtc,q}$ corresponding to the vertices of $\N(\dtc)$ are all $1$'s (the statement for $F_{\dtc}$ has been settled in \cite{Fc}).
\end{conjecture}

\subsection{Application to the Fock-Goncharov Duality Pairing}

We first briefly recall the Fock-Goncharnov's duality pairing \cite{FGc}.
Recall that a skew-symmetrizable matrix $B$ gives rise to a pair of cluster varieties $(\mc{A},\mc{X})$, 
and their {\em Langlands dual} $(\mc{A}^\vee,\mc{X}^\vee)$.
Fock-Goncharov duality conjecture \cite[Conjecture 4.1]{FGc} says that
the tropical points $\mc{X}^\vee(\mb{Z}^t)$ of $\mc{X}^\vee$ parametrize a basis of ring of regular functions $\mc{O}(\mc{A})$ of $\mc{A}$, and we can interchange the roles of $\mc{A}$ and $\mc{X}$.
The duality conjecture fails in general, but can hold with some mild assumption, or if we replace it with a formal version (see \cite{GHKK} for detail). 
From now on let us assume the duality conjecture holds, and denote the parametrizations by 
$$I_{\mc{A}}: \mc{A}(\mb{Z}^t) \hookrightarrow \mc{O}(\mc{X}^\vee)\ \text{ and }\ I_{\mc{X}^\vee}: \mc{X}^\vee(\mb{Z}^t) \hookrightarrow{} \mc{O}(\mc{A}).$$

The duality conjecture further asserts that we can require the parametrized bases to be {\em universally positive} and satisfy several interesting properties.
One of them concerns the pairing $$\mc{A}(\mb{Z}^t) \times \mc{X}^\vee(\mb{Z}^t) \to \mb{Z} .$$
There are two canonical ways to define this pairing:
\begin{align*} &&  I_{\mc{A}}(a)^{\op{trop}} (x)\ \text{ and }\   I_{\mc{X}^\vee}(x)^{\op{trop}} (a) &&  \text{for } a\in \mc{A}(\mb{Z}^t),\  x\in \mc{X}^\vee(\mb{Z}^t).
\end{align*}
The conjecture says that they are equal.
We are going to give a representation-theoretic interpretation of the above pairings in some special cases.
As a consequence, we shall see that the two ways of pairings are equal.
%The case we consider is when $B$ is skew-symmetric and $\det(B)=1$.
%The condition $\det(B)=1$ is sufficient and (conjecturally necessary \cite{L}) to ensure that $\mc{X}^\vee(\mb{Z}^t)$ and $\mc{A}(\mb{Z}^t)$ have the same dimension so that the pairing is nondegenerate.
Recall that there is a canonical map $\check{p}:\mc{A}^\vee \to \mc{X}^\vee$ given by $\check{p}^*(y_u) = \prod_v x_v^{B(v,u)}$,
where $\b{x}$ and $\b{y}$ are the coordinates of $\mc{A}^\vee$ and $\mc{X}^\vee$.
At the level of tropical points, this is given by $\mc{A}^\vee(\mb{Z}^t) \to \mc{X}^\vee(\mb{Z}^t),\ a \mapsto a B^{\T}$.
Note that if $B$ is invertible, then $\check{p}^*$ is injective.

As one can see immediately, the two pairings depend on the map $I_{\mc{A}}$ and $I_{\mc{X}^\vee}$.
According to Conjecture \ref{C:equalNewton}, this may not be an issue for the known interesting bases.
At this stage, let us first resolve this issue by letting $I_{\mc{A}}$ and $I_{\mc{X}^\vee}$ be the {\em generic basis map}.
More precisely, $I_{\mc{X}^\vee}$ and $I_{\mc{A}}$ are given by 
$$I_{\mc{X}^\vee}(\dtc) = \b{x}^{-\dtc} F_\dtc(\b{x})\ \text{ and }\ I_{\mc{A}}(a) = \b{y}^{-a} \check{F}_{\check{p}(a)}(\b{y}),$$ 
where $F_\dtc$ is the $F$-polynomial of $\Ker(\dtc)$ in the $\b{x}$-coordinate, and $\check{F}_\delta$ is the dual $F$-polynomial of $\Coker(\delta)$ in the $\b{y}$-coordinate.
%If $\dtc$ is not nonnegative, then by definition $F_{-e_i} = x_i$ and $\Fc_{-e_i} = y$?
The reason why we switch to the dual $F$-polynomial is due to the transposition of $B$ in the Langlands dual.

It is known that the $F$-polynomials may have negative coefficients so the usual tropicalization is not well-defined.
However, we can modify the usual tropicalization by considering the tropical $F$-polynomials.
Besides Remark \ref{R:sameNewton} this approach is further justified in \cite[Remark 1.4]{Fc}.
At least when the $F$-polynomial has positive coefficients, the two notions agree.  
So let us define 
$$I_{\mc{X}^\vee}(\dtc)^{\wtd{\op{trop}}} = f_\dtc\circ B^{\T} - \dtc\ \text{ and }\ I_{\mc{A}}(a)^{\wtd{\op{trop}}} = \fc_{\check{p}(a)} - a,$$
where $B^{\T}$ is the map of multiplication by the matrix $B^{\T}$.

\begin{theorem}[Fock-Goncharov duality pairing] \label{T:FGpairing} Suppose that $B$ is skew-symmetric.
The pairings $\mc{A}(\mb{Z}^t) \times \mc{X}^\vee(\mb{Z}^t) \to  \mb{Z}$ given by $I_{\mc{A}}(a)^{\wtd{\op{trop}}}(\dtc)$ and $I_{\mc{X}^\vee}(\dtc)^{\wtd{\op{trop}}}(a)$ are both equal to 
$\hom(aB^{\T},\dtc) - a\cdot \dtc$ in the following two situations
\begin{enumerate}
	\item The quiver of $B$ is mutation-equivalent to an acyclic quiver.
	\item Either $I_{\mc{X}^\vee}(\dtc)$ or $I_{\mc{A}}(aB^{\T})$ is a cluster variable, or equivalently either $\dtc$ or $aB^{\T}$ is negative reachable.
\end{enumerate}
\end{theorem}
%\marginpar{explain $\hom(\delta,\dtc)$}
\begin{proof}
%	\begin{align} \label{eq:FGpair1} I_{\mc{X}^\vee}(\dtc)^{\wtd{\op{trop}}}(p(\delta)) &= f_\dtc  (\delta B^{-1} B) - \delta B^{-1} \dtc^{\T} = \hom(\delta,\dtc) - \delta B^{-1} \dtc^{\T}; \\
%\label{eq:FGpair2}	I_{\mc{A}}(p(\delta))^{\wtd{\op{trop}}}(\dtc) &= \fc_{\delta} (\dtc) - \delta B^{-1} \dtc^{\T} = \hom(\delta,\dtc) - \delta B^{-1} \dtc^{\T}.
%	\end{align}
Due to Corollary \ref{C:HomEQP} for (1) and Theorem \ref{T:HomE} and \ref{T:HomEQP} for (2), we have that
	\begin{align} \label{eq:FGpair1} I_{\mc{X}^\vee}(\dtc)^{\wtd{\op{trop}}}(a) &= f_\dtc  (a B^{\T}) - a \cdot \dtc = \hom(aB^{\T},\dtc) - a\cdot \dtc; \\
	\label{eq:FGpair2}	I_{\mc{A}}(a)^{\wtd{\op{trop}}}(\dtc) &= \fc_{aB^{\T}} (\dtc) - a\cdot \dtc = \hom(aB^{\T},\dtc) - a\cdot \dtc.
	\end{align}
\end{proof}
 
\begin{remark} It is clear that Conjecture \ref{C:FGpairing} implies the equality of the two pairings in all skew-symmetric cases.	
	If $B$ is invertible, we can set $\delta=aB^{\T}$ and write $\hom(aB^{\T},\dtc) - a\cdot \dtc$ in a more symmetric form
	$$\hom(\delta,\dtc) + \delta B^{-1} \dtc^{\T}.$$
One can check this pairing is mutation-invariant using \cite[Proposition 6.1 and (2.11)]{DWZ2}.
	
Although the main part of Fock-Goncharov duality conjecture was intensively studied, the meaning of the duality pairing is only known in few cases. 
For the moduli space of the $\op{PGL}_2 / \SL_2$-local systems of surfaces, the duality pairing can be interpreted as the intersection pairing of laminations \cite[Proposition 12.1]{FG}.
The verification of the equality in this generality is new.
\end{remark}

%{The Principal Coefficient}
%Suppose $\bullet \leftarrow *$.\\
%$\br{\delta}=(\delta,0)$, $\br{\dtc}=(\dtc,-\dv)$;
%$c=(0,\delta), \check{c}=(-\dv, \delta)$.
%$$\b{x}^\delta F(\b{x})  (??)   =  \b{y}^{??} F(\b{y}) (\delta)$$

\section{Schur Representations and Dual Clusters} \label{S:dual}
\begin{definition}
	A representation $V$ is called {\em Schur} if $\Hom(V,V)=k$.
	It is called {\em real Schur} if in addition we require $\Ext^1(V,V)=0$.
\end{definition}

Here is a method to produce such $V$. We start with any representation $M$.
\begin{lemma} \label{L:Schur} Suppose that $\delta$ is $\E$-rigid such that $\hom(\delta,M)=1$.
	Then $t_\delta(M)$ is Schur. 	
	Dually,	suppose that $\dtc$ is $\Ec$-rigid such that $\hom(M,\dtc)=1$.
	Then $\fc_\delta(M)$ is Schur. 
	
	Moreover, if $M$ is $\Ec$-rigid (resp. $\E$-rigid), then $t_\delta(M)$ (resp. $\fc_\delta(M)$) is real Schur.
\end{lemma}

\begin{proof} By Corollary \ref{C:univhomo} $L=t_\delta(M)$ is the image of the nonzero homomorphism $C\to M$ where $C=\Coker(\delta)$.
	Since $L$ is a quotient of $C$, we have that 
	$$k \subseteq  \Hom(L,L) \subseteq \Hom(C,L) = k.$$
	Hence $\Hom(L,L)=k$.	
	If $M$ is $\Ec$-rigid, then by \cite[Proposition 4.8]{Fc} we have $\Ext^1(L,L)=0$.
	
	%	By Corollary \ref{C:univhomo} $L=\tc_\delta(M)$ is the kernel of the nonzero homomorphism $M\to K$ where $K=\Ker(\tau_p \delta)$.
	%	Since $N$ is a subrepresentation of $\tau C$, we have that 
	%	$$k \subseteq  \Hom(N,N) \subseteq \Hom(N,\tau C) = \E(C,N)= k.$$
	%	Hence $\Hom(N,N)=k$.
	%	By \cite[Lemma 4.1, Proposition 4.8]{Fc} we have $\Ext^1(N,N)=0$.
\end{proof}

%\subsection{Dual Clusters}
Let $d_0$ be an $\E$-rigid presentation with $\ind(d_0)=|Q_0|-1$.
Then by Proposition \ref{P:+-}, there are two complements $d_-$ and $d_+$ of $d_0$ satisfying
$\e(d_-,d_+)=e>0$ and $\e(d_+,d_-)=0$.
In this case we define the {\em sign} of $d_\pm$ in the cluster $\{d_\pm\}\cup\ind(d_0)$ to be $\pm$.
Throughout this section we will always assume that $e=1$. 
In other words, $(d_-,d_+)$ is a regular exchange pair.
\footnote{It is known that this assumption is always satisfied if the algebra is the Jacobian algebra of some quiver with generic potential and the cluster is (negative or positive) reachable.}
In this case, $d_-$ and $d_+$ fit into the triangle in $K^b(\proj A)$
$$d_+ \to d \to d_- \to d_+[1],$$
where $d\in \op{add}(d_0)$.
Let $\delta_0$ and $\delta_\pm$ be the weight vectors of $d_0$ and $d_\pm$ respectively.

Let $L=\Coker(d_+),\ N=\Coker(d_-),$ and $N^{\tau} = \Ker (\nu d_-)$.
Note that if $d_-$ is nonnegative (i.e., $\neq (P_i \to 0)$), then $N^\tau = \tau N$.
We have that $\hom(L,N^\tau)=\e(d_-,L)=1$.
We consider the exact sequence 
$$0\to K \to L \to N^\tau \to C \to 0,$$
where $L\to N^\tau$ spans $\Hom(L,N^\tau)$. 
Let $I$ be the image of $L \to N^\tau$.

\begin{definition} 
	We called $\dv I$, the {\em $c$-vector} of the exchange pair $(d_-,d_+)$.
	We also called $\pm \dv I$ the {\em signed $c$-vector} of $d_\pm$ for the cluster $\{d_\pm\} \cup \ind(d_0)$.
\end{definition}

According to Corollary \ref{C:univhomo}, we have that
$$K=\check{t}_{\delta_-}(L),\ \ C=f_{\delta_+}(N^\tau), \text{ and }\ I=\fc_{\delta_-}(L)=t_{\delta_+}(N^\tau).$$  

\begin{lemma} \label{L:duality} We have that  
	\begin{align*} \Hom(d_0, I)&=0,\ \E(d_0, I)=0;\\
	\Hom(d_+, I)&=k,\ \E(d_+, I)=0;\\
	\Hom(d_-, I)&=0,\ \E(d_-, I)=k.
	\end{align*}
	Moreover, $I$ is real Schur. 
\end{lemma}
\begin{proof} Since $\Hom(d_0\oplus d_-, N^\tau) = \E(d_-, d_0\oplus d_-)=0$ and $I$ is a subrepresentation of $N^\tau$, 
	we get $\Hom(d_0\oplus d_-, I) =0$.
	On the other hand, $\E(d_0\oplus d_+, L)=0$ and $I$ is a quotient of $L$, 
	so we have that $\E(d_0\oplus d_+, I) = 0$.
	By Lemma \ref{L:sub=h}, we have that 
	\begin{align*} \Hom(d_+,I) &= \Hom(d_+, N^\tau) = k;\\
	\E(d_-, I) &= \E(d_-, L) = k.
	\end{align*}
	Moreover, $I$ is real Schur follows from Lemma \ref{L:Schur}.
\end{proof}
\noindent We remark that Lemma \ref{L:sub=h} also tells us $\Hom(L,C)=0$ and $\E(L,\tau N)\cong \E(L,C)$.
Dually we have that $\E(N,K)=0$ and $\Hom(N,K)\cong \Hom(N,L)$.

Let $\bs{d}=\{d_1,d_2,\dots,d_n\}$ be a cluster of presentations, and $\bs{d}_j'=(\bs{d} \setminus \{d_j\})\cup \{d_j'\} $ be the adjacent cluster. Let $I_j$ be defined as above for each (unordered) exchange pair $\{d_j, d_j'\}$, and $\ep_j$ be the sign of $d_j$ in $\bs{d}$.
\begin{definition} 
For $I\in \rep A$ and a sign $\ep=\pm$, we denote $\ep I := \begin{cases} I & \text{if $\ep=+$} \\ I[1] & \text{if $\ep=-$} \end{cases}$ as an element in the bounded derived category $D^b(\rep A)$. 	
We define the {\em dual cluster} of $\bs{d}$ as the ordered elements $(\ep_1 I_1,\cdots,\ep_n I_n)$ in $D^b(\rep A)$.
\end{definition}
\noindent In this notation, we can rephrase Lemma \ref{L:duality} as 
$$\Hom(d_0, \pm I)=\Hom(d_-, I)=\Hom(d_+, -I)=0\ \text{ and }\ \Hom(d_-, -I)=\Hom(d_+, I)=k.$$

We use the upright $\updelta$ to denote the usual delta-function. We write $\delta^\perp$ for the abelian subcategory of $\rep A$
$$\delta^\perp := \{M\in \rep A \mid \hom(\delta,M) = \e(\delta,M) =0 \}.$$
\begin{theorem}[{{\em cf}. \cite[Lemma 5.3]{KY}}] \label{T:dualcluster} Let $\{\delta_i\}_{i}$ be a regular cluster and $I_j$ be defined as above. Then 
	\begin{equation}\label{eq:homdual} \hom(\delta_i, \ep_j I_j) = \updelta(i,j)\ \text{ and }\ \hom(\delta_i, -\ep_j I_j) = 0.	\end{equation}
	Moreover, the simple objects in category $\delta_I^\perp:=\bigcap_{i\in I} \delta_i^\perp$ are precisely $I_j$ for $j\notin I$.
	%Moreover, the simple objects in category $\delta_i^\perp$ are precisely $I_j$ ($j\neq i$),
	%and $I_j$ is the unique simple object in $\delta_{\hat{j}}^\perp$ where $\delta_{\hat{j}}=\sum_{i\neq j} \delta_i$.
\end{theorem}

\begin{proof} The first statement is a direct consequence of Lemma \ref{L:duality}.
	For the second statement, we already have that $I_j$ ($j\notin I$) are the simple objects in the category $\delta_I^\perp$.
	\cite[Theorem 3.8]{J} says that the category $\delta_I^\perp$ is equivalent to the module category of some (basic) algebra whose quiver has $|I|$ vertices less than $Q_0$.
	In particular, there are exactly $|Q_0|-|I|$ simple objects in $\delta_I^\perp$.
\end{proof}

\begin{remark} If we embed $K^b(\proj A)$ canonically into $D^b(\rep A)$, then the Euler form	
	$$\langle d,C\rangle=\sum (-1)^p \Hom_{D^b(\rep A)}\left(d,C[p]\right)\ \text{ on }\ K^b(\proj A)\times D^b(\rep A)$$
	gives us a non-degenerate pairing. %with the class of $P_v$ dual to the class of $S_v$. 
	This theorem shows in particular that the classes dual to the basis $\{\delta_i\}_i$ is given by
	$\{[\ep_i I_i]\}_i = \{\ep_i \dv I_i\}_i$.
				
	When $A$ is a finite dimensional Jacobian algebra associated to a nondegenerate QP $(Q,\mc{P})$,
the $c$-vectors in \cite{N} are defined as such dual basis.
For those reachable clusters, the $c$-vectors defined this way agree with the $c$-vectors of the corresponding clusters in the cluster algebra $\mc{C}(Q)$. This duality was further studied in skew-symmetrizable cases in \cite{NZ}.	
Here we gave an explicit construction of the real Schur representations corresponding to the $c$-vectors for any regular cluster. The {\em sign coherence} of the $c$-vectors is thus obvious from our construction.
\end{remark}

%\begin{remark} The result in \cite{} says that $\delta_i^\perp$ is isomorphic to the module category 
%	of $\End(\bigoplus_k T_k) / \innerprod{e_i}$ where $e_i$ is the idempotent in $\End(\bigoplus_k T_k)$ correspond to the summand $\End(T_i)$.
%\end{remark}

\begin{corollary} \label{C:pairing_mu} Suppose that we have the exchange triangle
		$$d_i^+ \to \bigoplus_j b_{ij}d_j \to d_i^- \to d_i^+[1].$$	
Let $\{\ep_jI_j\}_j$ $($resp. $\{\ep_j' I_j'\}_j$$)$ be the dual cluster of $\{d_1,\dots,d_i^+,\dots,d_n\}$ $($resp. $\{d_1,\dots,d_i^-,\dots,d_n\}$$)$.
Then $b_{ij} = \delta_i^-([\ep_j I_j]) = \hom(\delta_i^-, \ep_j I_j) = \delta_i^+ ([\ep_j' I_j']) = \hom(\delta_i^+, \ep_j' I_j')$, and $\delta_i^\pm([I_i]) = \pm 1$.
\end{corollary}
\begin{proof} We pair the triangle with the dual basis $[\ep_j I_j]$, and we obtain
	$$b_{ij} = \delta_i^-([\ep_j I_j])\ \text{ and }\ \delta_i^-([I_i]) = -1.$$
	If $\ep_j$ is positive, then $I_j$ is a quotient of $\Coker(\delta_j)$.
	We have that $\e(\delta_i^-, I_j) = 0$ so $b_{ij} = \hom(\delta_i^-, I_j)$.
	If $\ep_j$ is negative, then $I_j$ is a subrepresentation of $\Ker(\nu \delta_j)$.
	We have that $\hom(\delta_i^-, I_j) = 0$ so $b_{ij} = \e(\delta_i^-, I_j) = \hom(\delta_i^-, -I_j)$.
The rest can be proved similarly using the dual basis $[\ep_j' I_j']$.
\end{proof}

\begin{corollary} \label{C:CS} Let $\{\ep_i I_i\}_i$ be the dual cluster of $\{\delta_i\}_i$, and $M$ be the direct sum $\bigoplus_{i} I_i$.
	Then one of dual cone of $\N(M)$ is precisely spanned by this cluster. 
\end{corollary}
\begin{proof} 
%	\begin{equation} \label{eq:dc} \hom(\delta_i, I_j) = [\ep_j]_+ \updelta(i,j),\quad \e(\delta_i, I_j) = [\ep_j]_- \updelta(i,j).
%	\end{equation}
	Consider $M_\pm= \oplus_{\ep_i=\pm} I_i$, then $M_{\pm}$ is a vertex subrepresentation of $M$.
	We claim that $\F_{M_+}(M) = \innerprod{\delta_i}_i$.
	We first show that each $\delta_i\in \F_{M_+}(M)$, or equivalently
	$$\delta_i(\dv M_+) = \hom(\delta_i,M).$$
	But this is rather clear from \eqref{eq:homdual}.
	
	Next we show that each adjacent $\delta_i' \notin \F_{M_+}(M)$, or equivalently
	$$\delta_i'(\dv M_+) < \hom(\delta_i', M),$$
	which is equivalent to
	$$-\e(\delta_i', M_+) < \hom(\delta_i', M_-).$$
	But it is clear from Corollary \ref{C:pairing_mu} that 
	if $\ep_i >0$, then	$\hom(\delta_i', M_-)\geq 0$ and $\e(\delta_i', M_+)\geq 1$;
	if $\ep_i <0$, then	$\hom(\delta_i', M_-)\geq 1$ and $\e(\delta_i', M_+)\geq 0$.
\end{proof}

Finally, we pose some questions. Consider the following three sets consisting of \begin{enumerate}
	\item All real Schur representations;
	\item All real Schur representations constructed from Lemma \ref{L:Schur};
	\item All real Schur representations constructed from exchange pairs (Lemma \ref{L:duality}).
\end{enumerate}
It is clear that (1) contains (2), and (2) contains (3).
\begin{conjecture} For the finite-dimensional Jacobian algebras, the three sets are equal.
\end{conjecture}

%\begin{question} How to compute and $\hom$ and $\ext^1$ between two real Schur representations?
%\end{question}

\begin{problem} We say a set of real Schur representations is {\em compatible} if they are a part of some dual cluster.
Find some reasonable conditions without referring to the original cluster that can verify the compatibility.
\end{problem}

\section{The Dual Fan and the Edge Quiver} \label{S:fan&edge}
%\subsection{The General Case}
%\subsection{The Dual Fan $\F(M)$}
\begin{definition} A {\em fan} $\F$ in a real vector space $V$ is a finite collection of nonempty polyhedral cones in $V$ such that \begin{enumerate}
		\item every nonempty face of a cone in $\F$ is also a cone in $\F$;
        \item the intersection of any two cones in $\F$ is a face of both.
	\end{enumerate}
A fan is called {\em complete} if the union of all the cones in $\F$ is $V$.
\end{definition}

%\marginpar{dual cone of any face}

The dual cones of a polytope ${\sf P}$ fit together into a complete fan, the {\em dual fan} of $\sf P$.
It is also called the {\em normal fan} of $\sf P$.
To pedantically stick to the definition, we need the cones dual to faces (not just vertices) of $\sf P$.
Let ${\sf L}$ be a face of $\N(M)$. The dual cone $\F_{\sf L}(M)$ of $\sf L$ is spanned by
$$\{\delta\in \mb{Z}^{Q_0} \mid \delta(\gamma) = f_M(\delta),\ \forall \gamma \in {\sf L}\},$$
which is the intersection $\bigcap_\gamma \F_{\gamma}(M)$ over all vertices $\gamma \in {\sf L}$.
The dual cones of vertices are the maximal cones of the dual fan.
We denote the dual fan of $\N(M)$ by $\F(M)$.

A fan $\F_1$ is said to be a coarsening of a fan $\F_2$ if every cone of $\F_2$ is contained in some cone of $\F_1$.
A fan $\F_2$ is said to be a refinement of a fan $\F_1$ if every cone of $\F_1$ is a union of cones of $\F_2$.
If $\F_1$ is complete, then it is clear that $\F_2$ is a refinement of $\F_1$ then $\F_1$ is a coarsening of $\F_2$, but not vice versa. %\marginpar{rewrite}
It follows from Lemma \ref{L:refine} that
\begin{lemma}[{{\em cf}. \cite[Proposition 7.12]{Z}}] \label{L:fansum} Let $M_1$ and $M_2$ be any two representations of $A$.
	Then $\F(M_1\oplus M_2)$ is the common refinement of $\F(M_1)$ and $\F(M_2)$. 
\end{lemma}

Let us recall the {\em cluster fan} $\F^r(\rep A)$ of $A$ introduced in \cite{DF}.
The cones of $\F^r(\rep A)$ are spanned by $\{\delta_1,\dots,\delta_p\}$ such that each $\delta_i$ is real indecomposable and $\e(\delta_i,\delta_j)=0$ for $i\neq j$. 
Note that the maximal cones of $\F^r(\rep A)$ are precisely those spanned by the clusters.
%We are going to refine this simplicial fan as follows.
\begin{definition}
Let $\F(\rep A)$ be the set of all cones spanned by $\{\delta_1,\dots,\delta_p\}$ such that each $\delta_i$ is normal and $\e(\delta_i,\delta_j)=0$ for $i\neq j$.
By Theorem \ref{T:CDPHom} and Corollary \ref{C:indnormal}, $\F(\rep A)$ forms a simplicial fan as well. We call it {\em generalized cluster fan}. %\marginpar{more comments}
\end{definition}
\noindent It follows from Lemma \ref{L:clusterincone} that
\begin{proposition} \label{P:fanCC} The fan $\F(M)$ is a coarsening of the generalized cluster fan $\F(\rep A)$.
\end{proposition} 
\begin{remark} In view of Lemma \ref{L:fansum}, Propositions \ref{P:fanCC} and \ref{P:finitetype}, $\F(\rep A)$ can be viewed heuristically as the normal fan of the infinite dimensional representation $\bigoplus_{M \in \rep A} M$.
\end{remark}

%\subsection{The edge quiver of $\N(M)$}
Next we discuss the 1-skeleton of $\N(M)$. We will represent an edge of $\N(M)$, that is, an 1-dimensional face of $\N(M)$, by $\br{L_0 L_1}$ where $L_0$ and $L_1$ are vertex subrepresentations. Recall the functor $t_\dtb$ and $\tc_\dtb$ in Section \ref{S:tf}. 

%\begin{question} Is $L\cap M'$ a vertex subrepresentation of $M'$?
%If $L'$ is a vertex subrep of $M'$, then there is some vertex subrep $L$ of $M$ such that $L\cap M' = L'$. 
%\end{question}
		
\begin{proposition} \label{P:edgequiver} If $\br{L_- L_+}$ is an edge in $\N(M)$, then either $L_- \subset L_+$ or $L_+ \subset L_-$.
	Say $L_- \subset L_+$, then 
	$$L_- = t_\dtb(M)\ \text{ and } \ L_+ = \tc_\dtb(M)\ \text{ for any $\delta$ in the interior of $\F_{\br{L_- L_+}}(M)$}.$$
	Moreover, we have the following \begin{enumerate}
		\item $\delta_+(L_+/L_-)\geq 0$ for any $\delta_+ \in \F_{L_+}(M)$ and $\delta_-(L_+/L_-))\leq 0$ for any $\delta_- \in \F_{L_-}(M)$ with the equality holding only when $\delta_\pm \in \F_{\br{L_- L_+}}(M)$. 
		%\item For any two vertex subrepresentations $L\subset L'$, there is a path from $L$ to $L'$.
		\item If $\F_{L_-}(M)$ is spanned by a regular cluster, then $L_+/L_-$ is a direct sum of isomorphic real Schur representations.
	\end{enumerate}
\end{proposition}
\begin{proof} The convex hull of $\mc{L}(\delta,M)$ contains $\br{L_- L_+}$ for $\delta \in \F_{\br{L_- L_+}}(M) = \F_{L_-}(M) \cap \F_{L_+}(M)$.
	If $\delta$ is in the interior of $\F_{\br{L_- L_+}}(M)$, then $\delta \notin \F_{L}(M)$ for any other vertex $L$, so the convex hull of $\mc{L}(\delta,M)$ is exactly $\br{L_- L_+}$.
	By Theorem \ref{T:maxminsub} we have either $L_- \subset L_+$ or $L_+ \subset L_-$. 
	If $L_- \subset L_+$, then 
	$L_- = t_\dtb(M)\ \text{ and } \ L_+ = \tc_\dtb(M)$ for any $\delta\in \F_{\br{L_- L_+}}(M)$.

For (1), $L_+/L_- = \tc_\dtb(M)/t_\dtb(M)$ is $\delta$-semistable by Theorem \ref{T:maxminsub}. 
If $\delta_+ \in  \F_{L_+}(M) \setminus \F_{L_-}(M)$, then $\delta_+(L_-)<\delta_+(L_+) = f_M(\delta_+)$.
Hence $\dtb_+(L_+/L_-)> 0$.
Similarly, we get $\dtb_-(L_+/L_-)< 0$.	

For (2), suppose that $\F_{L_-}(M)$ is spanned by a regular cluster $\bs{\delta}$.
Then $L_+/L_-$ is $\delta$-semistable for any $\delta\in \F_{\br{L_-L_+}}(M)$.
There is only one element in $\bs{\delta}$ lying outside $\F_{\br{L_-L_+}}(M)$,
so by Theorem \ref{T:dualcluster} $L_+/L_-$ must be an iterated extension of a real Schur representation $E$.
But $\Ext^1(E,E)=0$, so it has to be a direct sum of $E$. 
\end{proof}

\begin{definition} We assign the orientation $L_0 \to L_1$ for each edge $\br{L_0 L_1}$ with $L_0\subset L_1$. We call the resulting oriented graph the {\em edge quiver} of $\N(M)$, denoted by $\N_1(M)$. We call $L_1/L_0$ an {\em edge factor} of $M$. 
\end{definition}

For any 2 consecutive arrows $L_0\to L_1 \to L_2$, we have that $\Hom(L_1/L_0, L_2/L_1)=0$. 
Indeed, by Proposition \ref{P:vertex} we have that $\Hom(L_1,M/L_1)=0$.
Since $L_1/L_0$ is a quotient of $L_1$ and $L_2/L_1$ is a subrepresentation of $M/L_1$,
we have that $\Hom(L_1/L_0,L_2/L_1)=0$.
However, there could be some homomorphism if the 2 arrows are not consecutive as shown in the following example.

\begin{example}
%	Consider the quiver $\cyclicfourone{1}{4}{2}{3}{a}{b}{c}$ with potential $abc$.
Consider the same quiver with potential as in Example \ref{ex:QP4}.
	Let $M=\Coker(-1,1,1,0)$. The Newton polytope of $M$ was computed in \cite[Example 6.10]{Fc}.
	There is a path $0\to S_3 \to S_{34} \to L \to M$ in $\N_1(M)$, where $L$ is the vertex subrepresentation such that $M/L=S_3$.
	We see that $\Hom(S_3/0, M/L)=k$.

The point of this example is that the filtration of $M$ given by a path from $0$ to $M$	in $\N_1(M)$ may  
not be the {\em Harder-Narasimhan filtration} associated to any stability condition.
\end{example}

\begin{definition} The {\em exchange quiver} of $A$ is the dual graph of $\F^r(\rep A)$ with orientation given by $\{\delta_-\}\cup \bs{\delta}_0 \to \{\delta_+\}\cup \bs{\delta}_0$ if $\e(\delta_-,\delta_+)>0$.
\end{definition}

In the end of this section, we state a conjecture relating the maximal paths in a general representation of quiver to the Schur sequences introduced in \cite{DW2}.
Recall from \cite{Ka1} that if $\rep_\alpha(Q)$ contains a Schur representation, then $\alpha$ is called a {\em Schur root}. $\alpha$ is called {\em real} if $\innerprod{\alpha,\alpha}=1$, otherwise it is called {\em imaginary}. It is also called {\em isotropic} if $\innerprod{\alpha,\alpha}=0$. We denote $L\perp N$ if $\hom(L,N)=\ext(L,N)=0$, and denote $\gamma \perp \beta$ if $\hom(\gamma,\beta)=\ext(\gamma,\beta)=0$.
If $\gamma \perp \beta$, then the number of $\gamma$-dimensional subrepresentations of a general $(\beta+\gamma)$-dimensional representation is finite. We denote this number by $\gamma \circ \beta$.
\begin{definition}[\cite{DW2}] We call two dimension vectors $\gamma$ and $\beta$ {\em strongly perpendicular} if $\gamma \circ \beta =1$. We denote this by $\gamma \pperp \beta$.
	A sequence $(\beta_1,\beta_2,\dots,\beta_s)$ of Schur root is called a {\em Schur sequence} if $\beta_i \pperp \beta_j$ for all $i<j$.
\end{definition}

\noindent The Schur sequence was introduced as a simultaneous generalization of exceptional sequences, sequences arising from the canonical decomposition and the stable decomposition in the quiver setting. It has interesting applications in the Schubert calculus. We refer readers to \cite{DW2} for more background and motivation.

Let $\mc{S}(\alpha)$ be the set of all Schur sequences $(\beta_1,\beta_2,\dots,\beta_r)$ (of any length) such that 
$\alpha$ is a positive integral combination $\alpha=\sum_{i=1}^r c_i \beta_i$ and $c_i=1$ whenever $\beta_i$ is not real or isotropic.
\begin{conjecture} \label{T:Schurseq} There is a bijection between $\mc{S}(\alpha)$ and the maximal paths in $\N_1(\alpha)$.
\end{conjecture}

\begin{example} Let $Q$ be the quiver $\twoone{1}{2}{3}$, and $\alpha$ be the dimension vector $(3,5,2)$.
	Except for zero and itself, $\N(M)$ has 4 vertex subrepresentations $L_1,L_2,L_3,L_4$ of dimension
	$(0,3,0), (0,0,2), (2,3,2)$, and $(0,5,2)$ respectively (see Example \ref{ex:genacyclic}).
We list all paths from $0$ to $M$ in $\N_1(M)$ together with the corresponding Schur sequences.
\begin{align*} &0\to M && (3,5,2)\\
&0\to L_1 \to M &&  e_2, (3,2,2)\\
&0\to L_3 \to M &&  (2,3,2), (1,2,0)\\
&0\to L_2 \to L_3 \to M && e_3, (2,3,0), (1,2,0) \\
&0\to L_1 \to L_4 \to M && e_2, (0,1,1), e_1\\
&0\to L_2 \to L_4 \to M && e_3, e_2, e_1
\end{align*}
\end{example}

\section{Examples} \label{S:example}
In this section, we give some concrete examples of $\F(M)$ and $\N_1(M)$.
There are at least two parameters that we can vary. One is the representation $M$, and the other is the algebra $A$.
%One major case of interest is when $A$ is the Jacobian algebra of a QP, but we leave that for another project \cite{DFW}. 
\subsection{The case when \texorpdfstring{$M=A$}{M=A}}
%In this subsection we consider the case when $M$ is the algebra $A$ as a module of itself.
\begin{lemma} \label{L:ideal} A vertex subrepresentation $I$ of $A$ is a two-sided ideal of $A$.
Dually, for a vertex quotient representation $I^*$ of $A^*$, $I$ is a two-sided ideal of $A$.
\end{lemma}
\begin{proof} Recall our convention that all modules are right.
	We need to show that $AI \subseteq I$.
	If not, there is some $a$ such that $aI \nsubseteq I$.
	By the Wedderburn-Malcev theorem, we can assume that $a$ is in the radical of $A$.
	Then $(1+a)I \nsubseteq I$ as well.
	But $1+a$ is invertible so $(1+a)I$ has the same dimension vector as $I$.
	This contradicts the fact that $I$ is a vertex subrepresentation (see Proposition \ref{P:vertex}).
\end{proof}

\begin{proposition} Assume that $\fc_A(\dtc) = \hom(A,\dtc)$. If $\dtc\in \Fc_{A/I}(A)$ and $\dc\in \IHom(\dtc)$, then $\dc$ is surjective with the same kernel after tensoring with $A/I$.
	Dually, assume that $f_{A^*}(\delta) = \hom(\delta, A^*)$. If $\delta\in \F_{A^*/I^*}(A^*)$ and $d\in \PHom(\delta)$, then $d$ becomes injective with the same cokernel when tensoring with $A/I$.
\end{proposition}
\begin{proof} $\dtc\in \check{\F}_{A/I}(A)$ implies that $\dtc(A/I) = \hom(A, \dtc)$. % = \dv \Ker(\dtc)$.
	By the dual of Lemma \ref{L:sub=h}, we have that $\hom(A/I, \dtc)=\hom(A,\dtc)$. 
	This implies that the kernel of $\dc$ does not change after tensoring with $A/I$.
Then $\dim \Ker(\dtc) = \hom(A/I, \dtc) = \dtc(A/I) = \dtc((A/I)^*),$
which implies that $\dtc$ is surjective after tensoring with $A/I$.
%For $\dc: I_+\to I_-$, this says that
% $\hom(A/I, I_+) \to \hom(A/I,I_-)$
\end{proof}

%\begin{remark} Since $\Hom(A,\dtc) = \Hom(\delta,\Ac)$, we have an obvious duality:
% If $\N(\Ac)$ is presented by $\delta (-) \leq h$, then $\Nc(A)$ is presented by $\dtc(-) \leq h$.
%\end{remark}

\begin{example} \label{ex:ab=0} Consider the quiver $\affineA{1}{2}{3}{a}{b}$ with relation $ab$.
Except for the two trivial ones, $\N(A)$ has 7 vertex subrepresentations as listed in the left column.
The middle and right columns are the corresponding ideals and dual cones.
\begin{align*}& (P_1,S_3,S_3) && \innerprod{e_1,e_3} && (e_1,-e_2,-e_3)\\
&(S_2\oplus S_3,P_2,S_3) && \innerprod{e_2,e_3} && (-e_1,e_2-e_1,-e_3) \\
&(S_3,S_3,S_3) && \innerprod{e_3} &&  (e_1,e_2-e_1,-e_3)\\
&(S_2,P_2,0) && \innerprod{e_2} && (-e_1,-e_2,e_3-e_2-e_1) \\
&(P_1,P_2,0) && \innerprod{e_1,e_2} && (e_1,-e_2,e_3-e_2-e_1) \\
&(P_1,0,0) && \innerprod{e_1} && (-e_1,e_2-e_1,e_3-e_1,e_3-e_2-e_1)\\
&(S_2,0,0) && \innerprod{a} && (e_1,e_3-e_2-e_1,e_3-e_1)
\end{align*}
\end{example} % fine

\subsection{Cluster-Finite Algebras}
\begin{definition} We call an algebra {\em cluster-finite} if it has only finitely many indecomposable $\E$-rigid representations.
\end{definition}
\noindent %It must contain finitely many indecomposable $\Ec$-rigid representations as well by \eqref{eq:eec}.
 A cluster-finite algebra may not be representation-finite. For example, the preprojective algebra of Dynkin type (other than $A_i$, $i<5$). We showed that the cluster fan of a cluster-finite algebra is complete (\cite[Proposition 6.1]{DF}).
 In particular, the cluster fan is the same as the generalized cluster fan.
\begin{proposition} \label{P:finitetype} Suppose that $A$ is cluster-finite. Let $M$ be the direct sum of all $\E$-rigid representations. 
	Then the dual fan $\F(M)$ is the cluster fan of $A$, and the edge quiver $\N_1(M)$ is the exchange quiver of $A$. 
\end{proposition}
\begin{proof} The claim about the dual fan is a direct consequence of Theorem \ref{T:Mcones} and the completeness of the cluster fan.
	So if $(\delta_-,\delta_+)$ is an exchange pair and $\delta_\pm \in \F_{L_\pm}(M)$, then there is an arrow between $L_-$ and $L_+$.
	We need to show that the arrow has the correct direction.
	By Proposition \ref{P:edgequiver} it suffices to show that $\delta^-(\dv L_0)< 0$ and $\delta^+(\dv L_0)> 0$ where $L_0=\dtb_0^\perp(M)$ for $\delta_0 \in \F_{\br{L_-L_+}}(M)$.
	%We have that $\delta_\pm(L_\pm) = \hom(\delta_\pm, M).$
	Apply $\Hom(-,L_0)$ to the exchange triangle $d_+ \to d_0 \to d_-^e \to d_+[1]$ of Proposition \ref{P:+-}, and we get the exact sequence
	$$0\to \Hom(e\delta_-, L_0)\to 0 \to \Hom(\delta_+, L_0) \to \E(e\delta_-, L_0) \to 0\to \E(\delta_+,L_0)\to 0.$$
	Hence $\delta^-(\dv L_0)< 0$ and $\delta^+(\dv L_0)> 0$.
\end{proof}
	
%\begin{proof}	
%It remains to show that if $(\delta_-,\delta_+)$ is an exchange pair, and $\delta_\pm \in \F_{L_\pm}(M)$, then $L_- \subseteq L_+$.
%Let $N$ be the kernel of a general element in $\IHom(-\delta_-)$, then $\hom(\delta_-,N)=\e(\delta_-,\delta_-)=0$ and $\hom(\delta_+,N)=\e(\delta_-,\delta_+)>0$.
%Since $N$ is a direct summand of $M$, we have that
%$L_-\cap N=0$ but $L_+\cap N \neq 0$.
%Hence $L_- \subseteq L_+$ by Proposition \ref{P:edgequiver}.
%\end{proof}

\begin{example} We continue with Example \ref{ex:ab=0}. There are 9 indecomposable representations of $A$.
Except for indecomposable projective, injective, and simple representations, they are $R=\Coker(1,-1,0)$ and $T=\Coker(1,1,-1)$. They are either $\E$-rigid or $\Ec$-rigid.
It turns out that to get the cluster fan of $A$, we do not need all of them as in Proposition \ref{P:finitetype}.
We have two minimal choices. One is $P_2,P_3,I_1,I_2,R,T$, and the other is $P_1,P_2,P_3,I_1,I_2,I_3$.
Here is the polytope for the first choice. We also display the edge quiver and the edge factors.
The 18 vertices correspond to the 18 clusters.
$$\polytopeone$$
The statement for the dual cluster in Corollary \ref{C:QPfinite} also holds here because we can check that each cluster is regular.
\end{example}

Let $(Q,\mc{P})$ be a {\em nondegenerate} quiver with potential such that its Jacobian algebra $A$ is finite-dimensional and cluster-finite.
Let $\mc{C}(Q)$ be the cluster algebra associated to the quiver $Q$.
Then the results in \cite{DWZ2} implies that the map sending each $\delta$ to the corresponding cluster variables induces an isomorphism from  the cluster fan of $A$ to the ordinary cluster fan of $\mc{C}(Q)$.

%\marginpar{birigid}
\begin{corollary} \label{C:QPfinite} Let $M$ be the direct sum of all $\E$-rigid representations of $A$.
Then the dual fan $\F(M)$ is the cluster fan of $\mc{C}(Q)$, and the edge quiver $\N_1(M)$ is the exchange quiver of $\mc{C}(Q)$. 
Moreover, the signed dimension vectors of the real Schur representations attached to the arrows from/to a fixed vertex $L$ are the signed $c$-vectors dual to the $\g$-vectors in the corresponding cluster.
\end{corollary}

\begin{remark}  We recover and generalize the main result in \cite{BDMTY}, where the authors obtain the similar result for Dynkin quivers (without potentials). In such cases, the Newton polytope is the so-called generalized associahedron \cite{FZy}.
	
We conjecture that any strict subset of $\ind(M)$ cannot do the job. 
More precisely, let $N$ be a direct sum of elements in any strict subset, then $\F(N)$ is {\em not} the cluster fan of $\mc{C}(Q)$.
We are able to prove this conjecture for the Dynkin quivers.
By contrast, we will see that for the preprojective algebras of Dynkin type, we need very few $\E$-rigid representations, namely projective ones only.
\end{remark}

\subsection{Preprojective Algebras}
In this subsection, we let $A$ be the preprojective algebra of a Dynkin diagram.
In \cite{BKT} three authors showed that if $M$ is a general representation in some irreducible component of $\rep_{\nu}(A)$, then $\N(M)$ is the {\em MV polytope} of certain basis element of $k[U]$ associated to $M$,
where $U$ is the maximal unipotent group of the simple, connected, simply-connected Lie group of the same Dynkin type.
This is also part of our motivation for studying the Newton polytope of a representation.

An interesting result in \cite{M} says that the maximal rigid presentations $d_w$ can be labelled by the elements $w$ in the Weyl group of the same Dynkin type. The cokernel of $d_w$ is the ideal $I_w$ of $A$ introduced in \cite{BIRS}.
\begin{proposition} \label{P:preproj} The vertices of $\N(A)$ are labelled by the ideals $I_w$, and $\F_{I_w}(A)$ is the cluster corresponding to $d_w$. 	
So $\F(A)$ is the cluster fan $\F(\rep A)$, which is a Weyl fan.
\end{proposition}

\begin{proof} %By Lemma \ref{L:ideal} each vertex subrepresentation of $A$ is an ideal, and by ?? it is $\E$-rigid ?
%So it suffices to show that each $I_w$ is some vertex subrepresentation of $A$.

Let $\delta_w$ be the weight vector of the maximal rigid presentation $d_w$.
We claim that $t_{\delta_w}(A) = I_w$, which implies that
$I_w$ is the vertex subrepresentation of $A$ such that
$\F_{I_w}(A)$ is the cluster corresponding to $d_w$.
It is known (e.g, \cite{BKT}) that the $I_w$ determines a torsion pair
$$\mc{T}(I_w)=\{M\in \rep A \mid \Ext^1(I_w,M)=0\}\ \text{ and }\ \mc{F}(I_w)=\{M\in \rep A \mid \Hom(I_w,M)=0\}.$$
On the other hand, recall from Definition \ref{D:torsion} that the torsion free class $\mc{F}(\delta_w)$ associated to $\delta_w$ is $\mc{F}(I_w)$ as well.
So its associated torsion class is $\mc{T}(I_w)$.
Now the claim follows from the the exact sequence $0 \to I_w \to A \to A/I_w \to 0$.
Indeed, from $I_w\in \mc{T}(I_w)$ and $A/I_w\in \mc{F}(I_w)$ \cite{BKT},
we conclude that $t_{\delta_w}(A)=I_w$.

Since $\F(A)$ is a coarsening of $\F(\rep A)$, we must have the equality $\F(A)=\F(\rep A)$, and thus there are no more vertices other than $I_w$.

%Firstly $\delta_w(\dv I_w) = \hom(I_w, I_w)$ and $\hom(\delta_w, A)=\hom(I_w, A)$.
%We have the exact sequence $0 \to I_w \to A \to I_{w_0w^{-1}}^* \to 0$ so
%$$0 \to \Hom(I_w,I_w) \to \Hom(I_w,A) \to  \Hom(I_w,I_{w_0 w^{-1}}^*) \to \E(I_w,I_w) = 0.$$
%
%Then compare $\delta_w(\dv I_w)$ and $\delta_w(\dv I_{w'})$.
\end{proof}

\begin{example} Let $T_{ij}$ be the indecomposable representation with socle $S_i$ and top $S_j$, and $R_2$ (resp. $R^2$) be the $(1,1,1)$-dimensional indecomposable representation with socle $S_2$ (resp. top $S_2$). We display the Newton polytope of $A$ for Dynkin type $A_3$. The vertices are labelled by the 24 permutations of the symmetric group $\mf{S}_4$. 
$$\polytopepreproj$$
\end{example}

\section*{Acknowledgement}
The author would like to thank Hugh Thomas for encouraging him to publish this manuscript.
He would like to thank LinHui Shen for explaining to him in detail the Fock-Goncharov duality pairing during the CRM conference Aug 2019.

\bibliographystyle{amsplain}

\end{document}